\documentclass{amsart}
\usepackage{lipsum}
\usepackage[dvips]{graphicx}
\usepackage{amssymb,latexsym}
\usepackage[ansinew]{inputenc}
\usepackage{mathrsfs}
\usepackage{psfrag}
\usepackage[all]{xy}
\usepackage{color}
\usepackage{graphicx,caption}
\usepackage{subfig}
\usepackage{hyperref}

\usepackage{comment}
\usepackage{amsfonts}
\usepackage{float}
\usepackage{amsmath,amssymb}
\usepackage{lscape}
\usepackage{tikz}
\usepackage{tikz-3dplot} 
\usetikzlibrary{arrows.meta,decorations.pathmorphing,patterns}

\usepackage{subfig}
\newtheorem{lemma}{Lemma}[section]
\newtheorem{corollary}{Corollary}[section]
\newtheorem{proposition}{Proposition}[section]
\newtheorem{theorem}{Theorem}

\newtheorem{remark}{Remark}[section]

\newcommand{\honourline}{%
  {\normalfont\itshape In memory of Professor Maria Aparecida Soares Ruas (1948--2025)}%
}

\newcommand{\R}{\mathbb{R}}

\def\qed{\hfill $\Box$}

\def\b0{\mbox{\boldmath $0$}}

\begin{document}

\title[Some Robust features of Surfaces in 4-space]{SOME REMARKS ON ROBUST FEATURES OF SURFACES IN 4-SPACE\\[0.5em]
\honourline}

\author{JORGE LUIZ DEOLINDO-SILVA}

 \thanks{Corresponding author: Jorge Luiz Deolindo-Silva
(\texttt{jorge.deolindo@ufu.br}).}

\address{Instituto de Ci\^encias Exatas e Naturais do Pontal, Universidade Federal de Uberl\^andia,
Ituiutaba, MG, Brazil}
\email{jorge.deolindo@ufu.br}



\begin{abstract} 

We study robust features of smooth surfaces in 4-space associated with contact with lines and hyperplanes. We investigate the local and multi-local singularities of orthogonal projections and height functions that describe these contacts and their corresponding geometric loci on the surface. At an $A_4$-singularity of a height function, we show that the $A_3$-curve and the $A_1A_2$-curve are generically tangent with contact of order $2$. We further establish criteria for the versality of the families of height functions and orthogonal projections in terms of the regularity of the $B_2$-curve and $S_2$-curve, respectively. Finally, we obtain the $[A_2]$- and $(A_0S_0)_2$-curves arising from multi-local singularities of orthogonal projections and describe their tangency properties at the $B_3$- and $H_3$-singularities.

\end{abstract}

\maketitle

\section{Introduction}\label{section1}

We study features of the extrinsic differential geometry of smooth surfaces in 4-space (Euclidean, affine, or projective) that arise from their contact with flat objects, such as hyperplanes and lines. This topic has been of interest for a long time and has received renewed impetus from the introduction of techniques from singularity theory over the past 15 years (see, for example, \cite{bruce-nogueira,BruceTari,Deolindo,Deolindo2,DM, Ronaldoetal,dmm,projsrfR4,OsetTari}).

The contact of a surface in 4-space with lines is captured by the singularities of orthogonal projections to 3-space, whereas the singularities of height functions capture the contact with hyperplanes. 
Certain types of singularities describing these contacts occur along curves on surfaces when the singularity has codimension $2$, and at isolated points when it has codimension $3$. These give rise to geometric features that can be followed under small deformations of the surface and are commonly referred to as {\it robust features}. Their geometry has been extensively studied; see, for instance, \cite[Chapter 7]{IFRT-book}. 
In particular, \cite{dmm} shows that the closure of the set of points on the surface at which the height function $h_{\bf v}$ has an $A_3$-singularity, the so-called  {\it $A_3$-curve}, is a smooth curve that meets the parabolic set tangentially, with the $A_4$-singularity of $h_{\bf v}$ occurring at an isolated point on this curve. Moreover, when the surface is orthogonally projected along an asymptotic direction by $P_{\bf u}$, it exhibits several robust features, namely an $S_2$-curve, a $B_2$-curve, and an $H_2$-curve, which are the sets of $S_2$-, $B_2$-, and $H_2$-singularities of $P_{\bf u}$, respectively, together with isolated points corresponding to the $S_3$-singularity of $P_{\bf u}$ on the $S_2$-curve, the $B_3$- and $C_3$-singularities on the $B_2$-curve, and the $H_3$- and $P_3(c)$-singularities on the $H_2$-curve (see \cite{bruce-nogueira,projsrfR4,IFRT-book} and Figure \ref{fig:SpecialC} for the configuration).

In \cite{Deolindo}, the robust curves arising from multi-local singularities of the orthogonal projection were studied in a neighborhood of the isolated $P_3(c)$-singularity. Several curves on the surface were obtained, and these curves are tangent to the parabolic set at the $P_3(c)$-singularity. However, the multi-local aspects in neighborhoods of the isolated singularities $B_3$, $S_3$, $C_3$, and $H_3$ are not described there. The same applies to the $A_4$-singularity of the height function. This paper aims to determine and characterize the multi-local curves on the surface arising from local and multi-local singularities in these mappings that have not been analyzed previously, as well as to study some geometric behavior of the $S_2$-curve and the $B_2$-curve in the hyperbolic region.

For our purposes, in Section \ref{adjacency} we complete the $\mathcal{A}$-adjacencies of the $B_3$-, $S_3$-, $C_3$-, and $H_3$-singularities to $\mathcal{A}_e$-codimension $2$ multi-local singularities, which are necessary for obtaining the corresponding multi-local curves on the surface. In fact, we show that neither the $C_3$- nor the $S_3$-singularity is adjacent to a multi-germ, whereas the $H_3$- and $B_3$-singularities are adjacent to the multi-germs $[A_2]$ and $(A_0S_0)_2$, respectively (see Figure \ref{fig:adjacencias_codim3_codim2}).

In Section \ref{robust}, we prove that the $A_3$-curve and the $A_1A_2$-curve (the set of $A_1A_2$-singularities of the height function) are generically tangent at the $A_4$-singularity of the height function, with contact of order $2$. Moreover, the family of height functions is a versal unfolding of the $A_4$-singularity at a point $p$ if and only if the $B_2$-curve is regular at $p$. Additionally, the $S_2$-curve is regular at a hyperbolic point $p$ if and only if the family of orthogonal projections is a versal unfolding of the $S_2$-singularity. Furthermore, the $S_2$-curve is transverse to one asymptotic direction, while the other asymptotic direction is tangent to the $S_2$-curve at the $S_3$-singularity.

Finally, we prove the existence of curves that are the loci of multi-local singularities of the orthogonal projections at the $H_3$- and $B_3$-singularities. The $[A_2]$-curve and the $(A_0S_0)_2$-curve (the loci of the corresponding $[A_2]$- and $(A_0S_0)_2$-singularities of the orthogonal projection) are obtained. The $[A_2]$-curve is generically tangent to the $B_2$-curve at the $B_3$-singularity, whereas the $(A_0S_0)_2$-curve is generically tangent to the parabolic set at the $H_3$-singularity. We give some preliminaries in Section \ref{section2}.



\section{Preliminary}\label{section2}

We give a brief review and establish some notation concerning a smooth surface $M$ in $\mathbb R^4$. 
Let $M$ be a regular surface in the Euclidean space $\mathbb R^4$. Consider a point $p\in M$ and a unit circle in $T_pM$ parametrized by $\theta \in [0,2\pi]$. The \emph{curvature ellipse} $\eta(\theta)$ in the normal plane $N_pM$, is the image of this unit circle described by a pair of quadratic forms $(Q_1, Q_2)$ (\cite{little}). Points on the surface are classified according to the position of the point $p$ with
respect to the ellipse ($N_pM$ is viewed as an affine plane through $p$). The point $p$ is
called \emph{elliptic/parabolic/hyperbolic} if it is inside/on/outside the ellipse at $p$, respectively.

The pair of quadratic forms is the $2$-jet of the
1-flat map $\mathcal F:{({\mathbb R}}^2,0)\to {{(\mathbb R}}^2,0)$ (i.e. without constant or linear terms) whose graph,
in orthogonal coordinates, is locally the surface $M$. Using a different approach to the geometry of surfaces in ${{\mathbb R}}^4$ given in \cite{bruce-nogueira}, each point on the surface determines a pair of quadratics:
$$ 
(Q_1,Q_2)=(ax^2+2bxy+cy^2,lx^2+2mxy+ny^2). 
$$

Representing a binary form $Ax^2+2Bxy+Cy^2$ by its coefficients $(A:B:C)\in {\mathbb R}P^2$, there
is a cone $\Gamma: B^2-AC=0$ representing the perfect squares.
If the forms $Q_1$ and $Q_2$ are linearly
independent, then they determine a line in the projective plane ${\mathbb R} P^2$ and the cone is a
conic. This line meets the conic $\Gamma$ in 0, 1 or 2 points according as
$\delta(p)<0,=0,>0$ where 
\begin{equation}\label{ParaCurve} 
\delta(p)=(an-cl)^2-4(am-bl)(bn-cm).
\end{equation}
Thus a point $p$ is {elliptic/parabolic/hyperbolic} if $\delta(p)<0/=0/>0$. The set of points in $M$
 where $\delta=0$ is called the \emph{parabolic set} of $M$ and is denoted by $\Delta$.
If $Q_1$ and $Q_2$ are dependent at a point $p$, then the point is called \emph{inflection point}.


Consider the action of the group ${\mathcal G} = GL(2, \mathbb{R}) \times GL(2, \mathbb{R})$ on pairs of binary forms. The classification of these ${\mathcal G}$-orbits is presented in Table \ref{Gibson} (see, e.g., \cite{gibson} for further details):

\begin{table}[h]
\caption{$GL(2,{\mathbb R})\times GL(2,{\mathbb R})$-classes of pairs of binary forms. } \label{Gibson}
\begin{tabular}{cc}
\hline
Normal form & Name\\
\hline
$(x^2,y^2)$ & \mbox{hyperbolic point} \\
$(xy,x^2-y^2)$ & \mbox{elliptic point}\\
$(x^2,xy)$ &\mbox{parabolic point}\\
$(x^2+ y^2,0)$ or $(xy,0)$ &\mbox{inflection point}\\
$(x^2,0)$ &\mbox{degenerate  inflection point}\\
$(0,0)$ &\mbox{degenerate inflection point}\\
\hline
\end{tabular}

\end{table}

The geometrical characterization of points on $M$ using singularity theory is carried out in \cite{dmm} via the family of height functions

\begin{equation}\label{Height}
\begin{array}{rcl}
h : M\times S^3& \to& {\mathbb R} \\
(p,{\bf v})&\mapsto&\langle p,{\bf v}\rangle.
\end{array}
\end{equation}
 where $S^3$ denotes the unit sphere in ${\mathbb R}^4$. For ${\bf v}$ fixed, the height function $h_{\bf v}(p)=h(p,{\bf v})$ is singular if and only
if ${\bf v}\in N_pM$. It is shown in \cite{dmm} that elliptic points are non-degenerate critical points of
$h_{\bf v}$ for any ${\bf v}\in N_pM$. At a hyperbolic point, there are exactly two directions in $N_pM$, labelled \emph{binormal directions},
such that $p$ is a degenerate critical point of the corresponding height functions. The two binormal directions coincide at a parabolic point. The direction of the kernel of the Hessian of the height functions along a binormal
direction is an \emph{asymptotic direction} associated to the given binormal direction (\cite{dmm}).
The asymptotic directions are labelled by conjugate directions in \cite{little}, and are defined as the directions along $\theta$ such that the curvature vector $\eta(\theta)$ is tangent to the curvature
ellipse (see also (\cite{Ronaldoetal, dmm})). This is invariant under the action of $GL(2,\mathbb R)$ in the normal plane, so the concept of asymptotic direction is affine invariant. 
So if $p$ is not an inflection point, then there are $2/1/0$ asymptotic directions at $p$ depending on $p$ being a hyperbolic/parabolic/elliptic point. If $p$ is an inflection point, then every direction in
$T_pM$ is asymptotic. 

Asymptotic directions can also be described via the singularities
of the orthogonal projections of $M$ to hyperplanes (see \cite{bruce-nogueira}). The family of orthogonal projections is given by
\begin{equation}\label{OrtProj}
\begin{array}{rcl}
P : M\times S^3& \to& TS^3 \\
(p,{\bf u})&\mapsto&({\bf u},p-\langle p,{\bf u}\rangle {\bf u}).
\end{array}
\end{equation}
For ${\bf u}$ fixed, the projection can be viewed locally at a point $p\in M$ as a map germ $P_{\bf u}:{(\mathbb R}^2,0)\to
({\mathbb R}^3,0)$. If we allow smooth changes of coordinates in the source and target
(i.e., consider the action of the Mather group $\mathcal{A}$) then the generic $\mathcal A$-singularities of $P_{\bf u}$
are those that have $\mathcal A_e$-codimension less than or equal to 3 (which is the dimension of $S^3$). These are listed in Table \ref{mond} (see  \cite{mond}) and a list of bi-germs and tri-germs is in Table \ref{bigerm} (see  \cite{kirk}).

\begin{table}[htp]
  \caption{The generic local singularities of projections of $M$ to hyperplanes.}\label{mond}
\begin{center}
\begin{tabular}{c|c|c}
\hline
Name & Normal form & ${\mathcal A}_e$-codimension\cr
\hline
\hline
Immersion & $(x,y,0)$ &0\cr
Cross-cap & $(x,y^2,xy)$ &0\cr
$B^{\pm}_{k}$ & $(x, y^2, x^2y \pm y^{2k+1}), k=2,3$ & $k$ \cr
$S^{\pm}_{k}$ & $(x, y^2, y^3 \pm x^{k+1}y), k=1,2,3$ & $k$\cr
$C^{\pm}_{k}$ & $(x, y^2, xy^3 \pm x^{k}y), k=3$ & $k$\cr
$H_k$ & $(x,xy+y^{2k+2},y^3), k=2,3$ & $k$\cr
$P_3(c)$ & $(x, xy + y^3, xy^2 + cy^4),\;c\neq0,\frac{1}{2},1,\frac{3}{2}$ & $3^*$\cr
\hline
\end{tabular}

$(*)$: the codimension in brackets is that is of the $\mathcal A$-stratum.

\end{center}
\end{table}

The projection $P_{\bf u}$ is singular at $p$ if and only if ${\bf u}\in T_pM$. The singularity is a cross-cap unless ${\bf u}$ is an asymptotic direction at $p$. As $P$ has  3-parameters, the ${\mathcal A}_e$-codimension $2$ singularities occur generically on curves on the surface, and the ${\mathcal A}_e$-codimension $3$ ones occur at special points on these curves (\cite{IFRT-book}). The closure of the set of points $p$ on $M$ for which there exists a projection $P_{\bf u}$ having an $S_2$ (resp. $B_2$, $H_2$, $[A_2]$, $(A_0S_0)_2$, $A_0S_1^{\pm}$, and $A_0S_0|A_1^{\pm}$)-singularity at $p$ is called by {\it $S_2$-curve} (resp. {\it $B_2$-curve, $H_2$-curve, $[A_2]$-curve, $(A_0S_0)_2$-curve, $A_0S_1^{\pm}$-curve}, and {\it $A_0S_0|A_1^{\pm}$-curve}). The $H_2$-curve coincides with the parabolic set (\cite{bruce-nogueira}). When ${\bf u}$ is an asymptotic direction, the $B_2$-curve, is also the $A_3$-curve 
along the binormal direction associated to ${\bf u}$ 
(\cite{bruce-nogueira}). The $B_2$-curve, $S_2$-curve, $(A_0S_0)_2$-curve, $A_0S_1^{\pm}$-curve, and $A_0S_0|A_1^{\pm}$-curve  meet the parabolic set tangentially at a $P_3(c)$-singularity (
\cite{Deolindo}) 
 and $B_2$-curve intersects the $S_2$-curve transversally at a $C_3$-singularity (see Figure \ref{fig:SpecialC} for their configurations and \cite{BruceTari}). At inflection points, the parabolic set has a Morse singularity. For more details on inflection points, see \cite{bruce-nogueira,Ronaldoetal}.

\begin{table}[h]
\caption{ Bi-germs and tri-germs of $\mathcal{A}_e$-codimension 2 of projection of $M$ to hyperplanes.}\label{bigerm}
\begin{tabular}{c|c|c}
\hline
Name & Normal Form &${\mathcal A}_e$-codimension\cr
\hline
$[A_2]$ & $(x,y,0;X,Y,X^2+Y^3)$ &$2$\cr
$(A_0S_0)_2$ & $(x,y,0;Y^2,XY+Y^5,X)$ &$2$\cr
$A_0S_1^{\pm}$ & $(x, y,0;Y^3\pm X^2Y,Y^2,X)$ & $2$ \cr
$A_0S_0|A_1^{\pm}$ & $(x, y,0, X,XY,Y^2 \pm X^2)$ & $ 2$ \cr
\hline
$A_0^3|A_2$ & $(x,y,0;\overline{x},0,\overline{y};X,Y,Y+ X^3)$ &$2$\cr
$(A_0^2|A_1^{\pm})A_0$ & $(x,y,0;\overline{x},0,\overline{y};X,Y,X^2\pm Y^2)$ &$2$\cr
\hline
\end{tabular}
\end{table}

\begin{figure}[htb]
\includegraphics[ width=9.5cm, height=6cm]{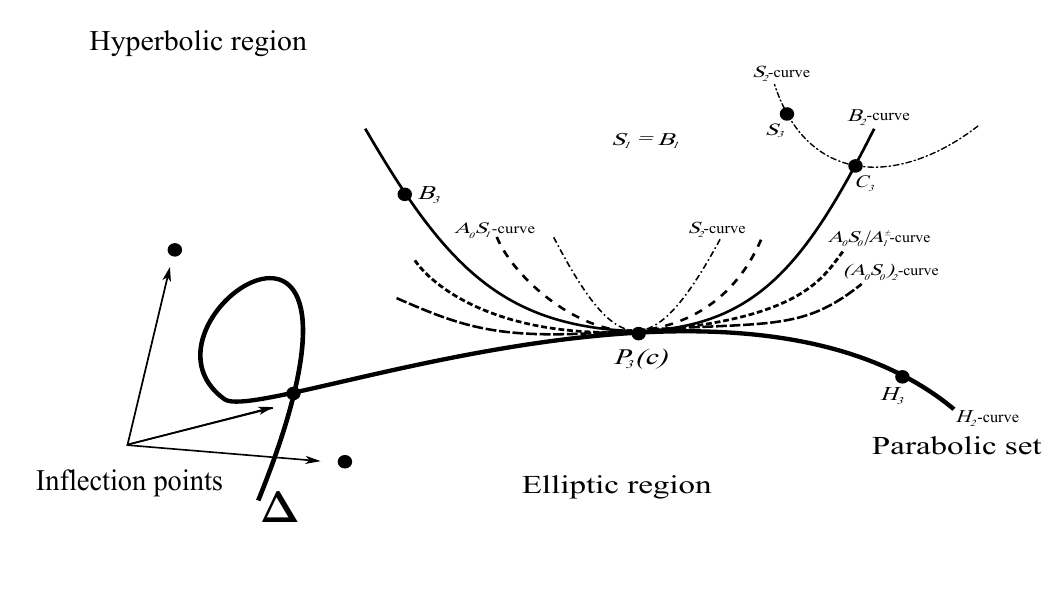}
\caption{Curves and special  points on $M$ (\cite{dmm,projsrfR4,Deolindo}).}
\label{fig:SpecialC}
\end{figure}

\section{Adjacency}\label{adjacency}

The analysis of the loci of multi-local singularities of the height function $h_{\bf v}$ and the orthogonal projections $P_{\bf u}$ near isolated points of the surface $M$ corresponding to codimension $3$ singularities naturally leads to the study of adjacencies from codimension $3$ local singularities to codimension $2$ multi-local singularities. In this section, we determine the $\mathcal{A}$-adjacencies of the $B_3$, $S_3$, $C_3$, and $H_3$ singularities to $\mathcal{A}_e$-codimension $2$ multi-local singularities that will be used in Section \ref{robust}.

We first recall some basic concepts from singularity theory. For further details, see \cite{IFRT-book}. An $s$-parameter {\it unfolding} of a map-germ
$f:(\mathbb{R}^2,0)\rightarrow(\mathbb{R}^k,0)$ is a map-germ
$$
\begin{array}{cccl}
    {F}: &(\mathbb{R}^2\times\mathbb{R}^s,0)  & \longrightarrow & (\mathbb{R}^k\times\mathbb{R}^s,0) \\
     & (x,y,{\bf u}) & \longmapsto & (\bar{f}(x,y,{\bf u}),{\bf u})
  \end{array}
$$
such that $\bar{f}(x,y,0)=f(x,y)$ and $f_{\bf u}(x,y)=\bar{f}(x,y,{\bf u})$. The germ $f_{\bf u}$ is called an $s$-parameter {\it deformation} of $f$. 
Two unfoldings $F,F'$ of $f$ are {\it $\mathcal{G}$-equivalent}
($\mathcal{G}=\mathcal{R},\mathcal{A}$)
if there exists a morphism 
$r:(\mathbb{R}^s,0)\to(\mathcal{G},I)$ with $I$ the identity in $\mathcal{G}$,
such that, for ${\bf u}\in(\mathbb{R}^s,0)$,
$f'_{\bf u}=r({\bf u})\cdot f_{\bf u}$. A $c$-parameter unfolding $F'$ of $f$ is {\it induced} from an
$s$-parameter unfolding $F$ by a germ
$\xi:(\mathbb{R}^c,0)\to(\mathbb{R}^s,0)$ if
$
(f'(x,y,{\bf v}),{\bf v})=(f(x,y,\xi({\bf v})),{\bf v})
$,
and an unfolding $F$ of $f$ is {\it $\mathcal{G}$-versal} if every other unfolding of $f$ is $\mathcal{G}$-equivalent to an unfolding induced from $F$. A $\mathcal{G}$-versal unfolding of $f$ therefore represents,
up to $\mathcal{G}$-equivalence, every other unfolding of $f$. The least number of parameters for a $\mathcal{G}$-versal unfolding is the $\mathcal{G}$-codimension of $f$. In that case, a versal unfolding is called {\it miniversal}. 
An analogous definition can be made for the extended group $\mathcal{G}_e$ by
substituting $\mathcal{G}$ with $\mathcal{G}_e$.
The $\mathcal{G}_e$-miniversal unfoldings $F$ of map-germs
$f:(\mathbb{R}^2,0)\to(\mathbb{R}^k,0)$
can be seen in \cite{Arnoldeal,IFRT-book} and \cite{GibsonMarar} for
$k=1$ and $k=3$, respectively.


Let $\mathcal{X}$ and $\mathcal{Y}$ be two classes of singularities. We say that $\mathcal{X}$ is {\it adjacent} to $\mathcal{Y}$, denoted by $\mathcal{X}\to\mathcal{Y}$, if, for any representative $f\colon(\mathbb{R}^2,0)\to(\mathbb{R}^k,0)$ of $\mathcal{X}$, there exists an unfolding $F(x,y,\mathbf{u})=(f_{\mathbf{u}}(x,y),\mathbf{u})$, with $f_0=f$, such that the set of points $(x,y,\mathbf{u})$ for which the germ of $f_{\mathbf{u}}$ at $(x,y)$ belongs to the class $\mathcal{Y}$ has $(0,0,0)$ in its closure.

For $k=1$, the $\mathcal{R}$-classification and $\mathcal{R}$-adjacencies of map-germs of $(\mathbb{R}^2,0)\to(\mathbb{R},0)$ were established in \cite{Arnoldeal}. In particular, it is known that $A_4\to A_3$ and $A_4\to A_1A_2$.

For the case $k=3$, the $\mathcal{A}$-adjacencies from mono-germs to mono-germs were determined by Mond in \cite{mond}, and the $\mathcal{A}$-classification of multi-germs was established in \cite{kirk}. The $\mathcal{A}$-adjacencies from a mono-germ to multi-germs have been little explored. 
For instance, Mond proved in \cite{mond} the adjacency
$B_3 \rightarrow [A_2]$
and conjectured the adjacency
$H_3 \rightarrow (A_0S_0)_2.$
Meanwhile, in \cite{Deolindo} it was shown that
$
P_3(c)\rightarrow (A_0S_0)_2,\;A_0S_1^{\pm},\;A_0S_0|A_1^{\pm}.
$
However, the adjacencies from the singularities $B_3$, $S_3$, $C_3$, and $H_3$ to $\mathcal{A}_e$-codimension $2$ multi-germs remain to be determined.


The aim here is to establish the remaining adjacencies, namely, whether the singularities $B_3$, $C_3$, $H_3$, and $S_3$ are adjacent to any of the multi-germs listed in Table \ref{bigerm}, thereby confirming Mond's conjecture.

\begin{remark}{\normalfont
\begin{itemize}
\item[(i)] The remaining multi-germs of $(\mathbb{R}^2,S)\to(\mathbb{R}^3,0)$ are omitted from Table~\ref{bigerm}, since unfoldings of the singularities $B_3$, $C_3$, and $S_3$ do not admit triple points (\cite{GibsonMarar}). Therefore, these singularities cannot be adjacent to any $\mathcal{A}_e$-codimension $2$ tri-germs, four-germs, or five-germs.
Furthermore, since $H_3$ is not adjacent to $\mathcal{A}_e$-codimension $2$ tri-germs (see Proposition \ref{H_3not}), it cannot be adjacent to $\mathcal{A}_e$-codimension $2$ four-germs or five-germs.
\end{itemize}
}
\end{remark}

The geometric criteria for the multi-germs listed in Table \ref{bigerm}, as well as their geometric interpretation (Figure \ref{figMulti}), can be found in \cite{kirk} and are described as follows:
\begin{itemize}

\item[(a)] $[A_2]$: {\it two regular surfaces with tangency of type $A_2$.}

\item[(b)] $(A_0S_0)_2$: {\it regular surface transversal to a cross-cap but tangent to the double points curve at the cross-cap point.}

\item[(c)] $A_0S_1^{\pm}$: {\it surface with singularity $S_1^{\pm}$ intersecting transversally another regular surface.}

\item[(d)]  $A_0S_0|A_1^{\pm}$: {\it regular surface tangent to a cross-cap but transversal to the double points curve at the cross-cap point.}

\item[(e)] $A_0^3|A_2$: {\it three transverse regular surfaces.}

\item[(f)] $(A_0^2|A_1^\pm)A_0$: {\it three regular surfaces, two of which have
$A_1$ tangency and the other is transverse to both.}
\end{itemize}

\begin{figure}[!htb]
\centering
\quad
\subfloat[{[$A_2$]}]{
\includegraphics[scale=0.3]{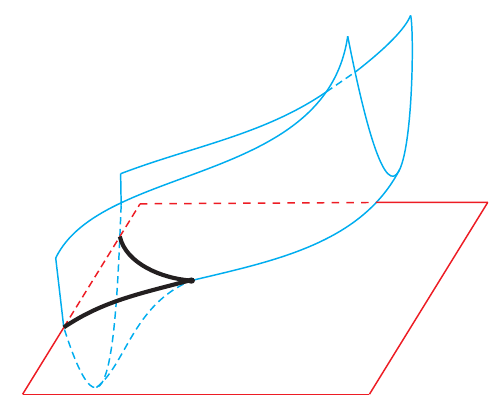}
}
\quad \quad 
\subfloat[$(A_0S_0)_2$]{
\includegraphics[scale=0.3]{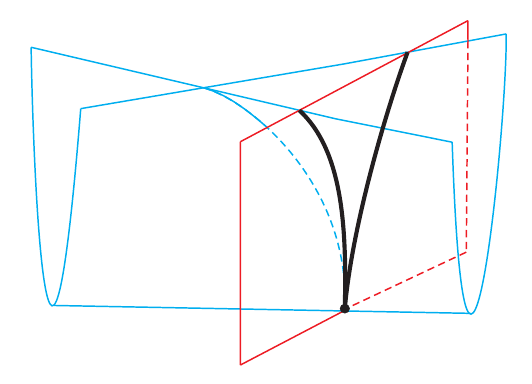}
}
\quad \quad 
\subfloat[$A_0S_1^{\pm}$]{
\includegraphics[scale=0.3]{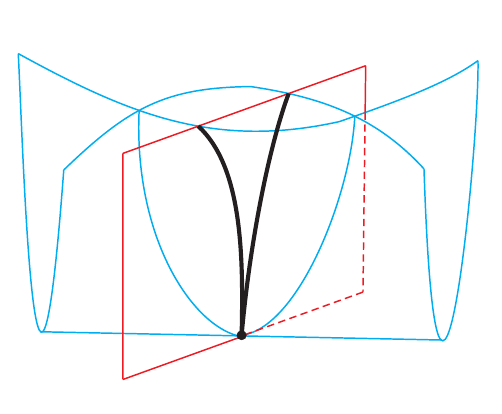}
}
\quad \quad   
\subfloat[$A_0S_0|A_1^{\pm}$]{
\includegraphics[scale=0.3]{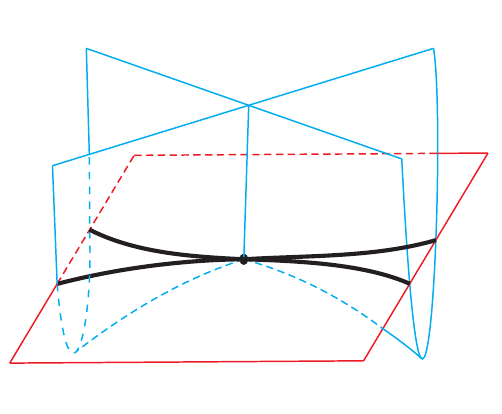}
}
\quad \quad \quad \quad
\subfloat [$A_0^3|A_2$]{
\includegraphics[height=2.0cm,width=3cm]{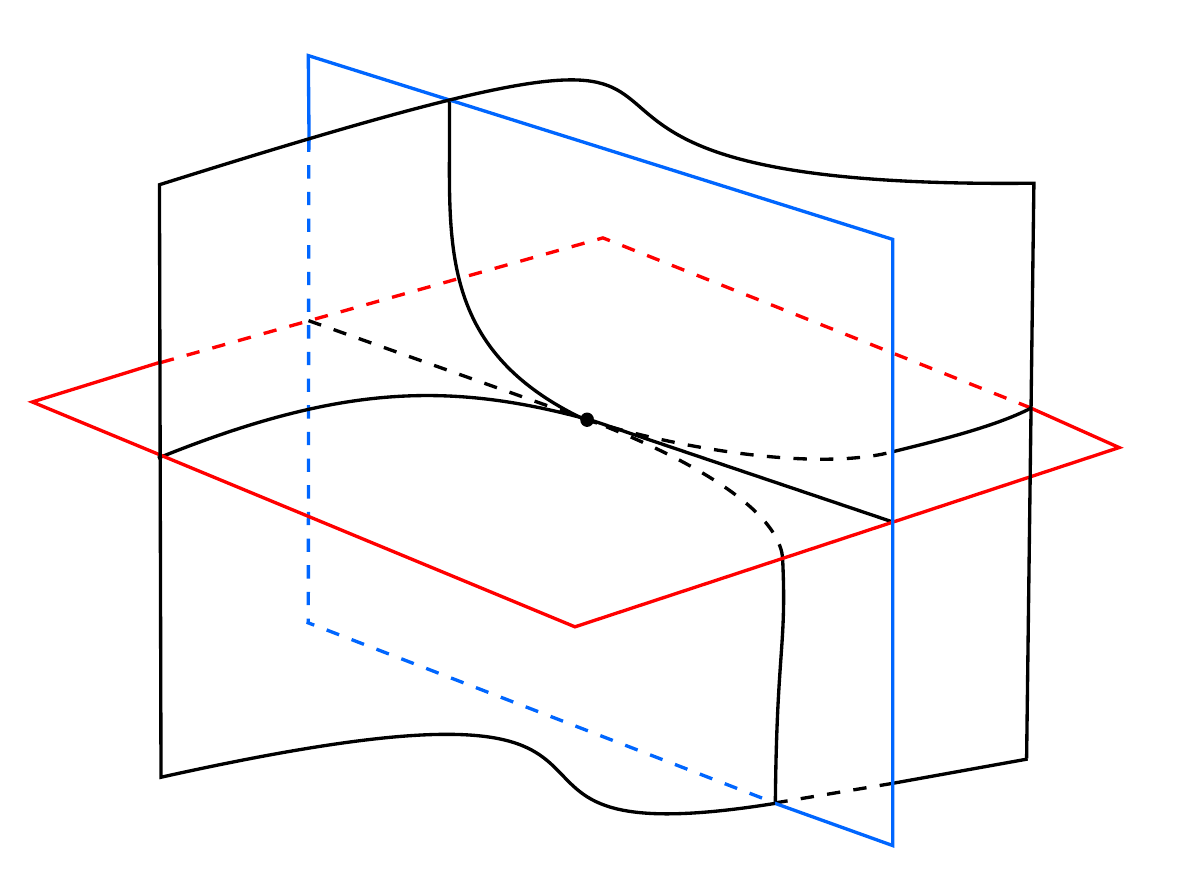}
}
\quad 
\subfloat[$(A_0^2|A_1^{\pm})A_0$]{
\includegraphics[height=2.0cm,width=3cm]{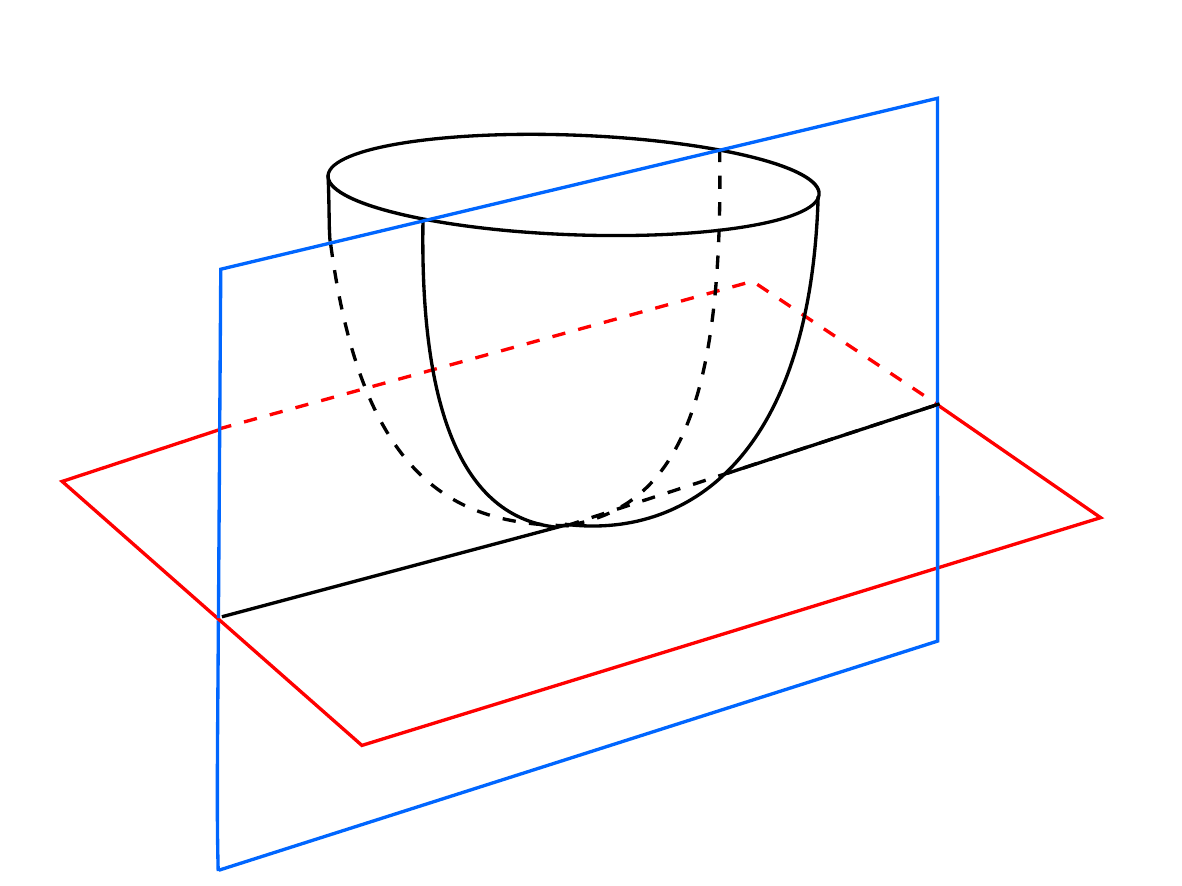}
}
\caption{Geometry of the bi-germs and tri-germs of  $\mathcal{A}_e$-codimension 2.}\label{figMulti}
\end{figure}











We have the following observation.

\begin{proposition}\label{desdobra}
The classes of singularities $B_3$, $C_3$, $S_3$ are not adjacent to bi-germs $(A_0S_0)_2$, $A_0S_1^{\pm}$, $A_0S_0|A_1^{\pm}$.
\end{proposition}

\begin{proof} 
To prove the non-adjacency to the bi-germs, we write the germs with singularities 
$B_3$, $C_3$, and $S_3$ in the normal form
$
(x,y^2,yp(x,y^2))
$
(see \cite{mond}) and consider their $\mathcal A_e$-miniversal deformations given by
$
F_{\bf u}(x,y)=(x,y^2,yp_{\bf u}(x,y^2)),
$
where $p_{\bf u}=p(x,y^2,{\bf u})$ is defined as follows:
$$
\begin{array}{l}  
B_3: \;\;x^2\pm y^6+a+by^2+cy^4,  \\  
C_3:\;\;xy^2\pm x^3+a+by^2+cx, \\  
S_3:\;\; y^2\pm x^4+a+bx+cx^2,
\end{array}
$$
with ${\bf u}=(a,b,c)$ (see \cite{GibsonMarar} for the deformations). Suppose that for ${\bf u}$ fixed there exist distinct points $(x_1,y_1)$ and $(x_2,y_2)$ such that 
$F_{\bf u}$ is regular at $(x_1,y_1)$ and singular at $(x_2,y_2)$, with
$F_{\bf u}(x_1,y_1)=F_{\bf u}(x_2,y_2)$. It follows that
$x_1=x_2=x$ and $y_1=-y_2$, with $y_1\neq0$.
Note that
$
F_x=(1,0,yp_x)
$
and
$
F_y=(0,2y,p+yp_y).
$
Therefore, $F_{\bf u}$ is singular at $(x_1,y_1)$ if and only if 
$y_1=0$ and $p(x_1,0)=0$.
However, this contradicts the fact that $y_1\neq0$. Hence, 
$B_3,\;C_3,$ and $S_3$ cannot be $\mathcal A$-adjacent to any of the bi-germs described in the proposition. \qed
\end{proof}

\begin{lemma}\label{cond [A2]}
Let $f(x,y)=(x,y^2,yp(x,y))$ be the normal form of a codimension $3$ map-germ
$f:(\mathbb R^2,0)\to(\mathbb R^3,0)$ and let 
$F(x,y,{\bf u})=(x,y^2,yp(x,y^2,{\bf u}))$ be its $\mathcal A_e$-miniversal deformation. 
The germ $f$ is $\mathcal A$-adjacent to the bi-germ $[A_2]$ if and only if the following equations are satisfied:
$$
p(x,y^2,{\bf u})=p_x(x,y^2,{\bf u})=p_y(x,y^2,{\bf u})=[p_{xx}p_{yy}-p_{xy}^2](x,y^2,{\bf u})=0.
$$
\end{lemma}

\begin{proof}
For {\bf u} fixed consider two distinct points $(x_1,y_1)$ and $(x_2,y_2)$. We have 
$F_{\bf u}(x_1,y_1)=F_{\bf u}(x_2,y_2)$ if and only if 
$x_1=x_2=x$, $y_1=-y_2$ and $p_{\bf u}(x,y_1^2)=0$. We still denote $F$ and $p$, instead $F_{\bf u}$ and $p_{\bf u}$.
Now, observe that
$
F_x=(1,0,yp_x),$ $F_y=(0,2y,p+2y^2p_y),$
$F_{xx}=(0,0,yp_{xx}),$ $F_{yy}=(0,2,2yp_y+4y(p_y+y^2p_{yy}))$
and
$
F_{xy}=(0,0,p_x+2y^2p_{xy}).
$
The equation of the tangent plane, in the coordinates $(X,Y,Z)$, at the point
$(x,y_1)$ is given by
$$
y_1p_xX+\frac{(p+2y_1^2p_y)}{2y_1}Y-Z=0.
$$
Since $y_1=-y_2$, the tangent plane at $(x,y_2)$ is
$
-y_1p_xX-\frac{(2y_1^2p_y)}{2y_1}Y-Z=0.
$
Therefore, the two tangent planes at $(x,y_1)$ and $(x,y_2)$ coincide if and only if
$p_x=p_y=0$.
Moreover, the order $2$ contact between the surfaces is given in \cite{OsetTari} by the relation
$$
\det(F_x,F_y,F_{xx})\cdot \det(F_x,F_y,F_{yy})
-(\det(F_x,F_y,F_{xy}))^2=0.
$$
We can write
$
\det(F_x,F_y,F_{xx})=2y^2p_{xx},
$
$
\det(F_x,F_y,F_{yy})=8y^4p_{yy}+8y^2p_y-2p
$
and
$
\det(F_x,F_y,F_{xy})
=4y^2(p_x^2+4y^2p_xp_{xy}+4y^4p_{xy}^2).
$
Hence, by taking $p=p_x=p_y=0$, the previous relation vanishes if and only if
$
p_{xx}p_{yy}-p_{xy}^2=0.
$
\qed
\end{proof}


\begin{proposition}
The classes of singularities $C_3$ and $S_3$ are not adjacent to the bi-germ $[A_2]$.
\end{proposition}

\begin{proof}
Consider the $\mathcal A_e$-miniversal deformation
$
F(x,y,{\bf u})=(x,y^2,yp(x,y^2,{\bf u}))
$
of $C_3$ and $S_3$, as in Proposition \ref{desdobra}. 
Consider the points $(x_1,y_1)\neq(x_2,y_2)$ such that
$F(x_1,y_1,{\bf u})=F(x_2,y_2,{\bf u})$. The proof follows from Lemma \ref{cond [A2]}.
%
\qed
\end{proof}

\medskip

To study the bi-germ $A_0S_1^{\pm}$, we need the recognition criteria for the singularities $S_0$ and $S_1^{\pm}$ given by Saji (\cite{Saji}).
%
Let $f:(\mathbb{R}^2,0)\rightarrow(\mathbb{R}^3,0)$ be a map-germ and $0$ a corank one singular point. Then there exist vector fields $\xi$ and $\eta$ near the origin such that $df_0(\eta_0)=0$ and $\xi_0,\eta_0\in T_0\mathbb{R}^2$ are linearly independent, where $(\xi_0,\eta_0)$ denotes the vector fields at the origin. Define $\varphi=det(\xi f,\eta f,\eta\eta f)$.
\begin{itemize}
\item[(i)] The germ $f$ is a cross-cap ($S_0$-singularity) if $\xi\varphi\neq0$.
\item[(ii)] If $\xi\varphi=0$, then $f$ at $0$ is $\mathcal{A}$-equivalent to $S_1^-$ if and only if $\varphi$ has a critical point at $0$ and $detHess \varphi(0) > 0$. On the other hand, $f$ at $0$ is $\mathcal{A}$-equivalent to $S_1^+$ if and only if $\varphi$ has a critical point at $0$, $detHess \varphi(0) < 0$, and the vectors $\xi f(0)$ and $\eta\eta f(0)$ are linearly independent.
\end{itemize}

Based on this criterion, we have the following result:

\begin{proposition}\label{H_3not}  
The singularity $H_3$ is not adjacent to the bi-germs $A_0S_1^{\pm}$, $[A_2]$, $A_0S_0|A^{\pm}_1$, $A_0^3|A_2$ and $(A_0^2|A_1^{\pm})A_0$. 
\end{proposition}

\begin{proof} 
Consider a $\mathcal A_e$-miniversal deformation of $H_3$ given by 
$$
F(x,y,{\bf u})=(x,y^3+ay,xy+y^8+by^2+cy^5),
$$ 
where ${\bf u}=(a,b,c)$ (see \cite{GibsonMarar}).

Suppose that $H_3\rightarrow A_0S_1^{\pm}$. One of the branches of the bi-germ has the singularity $S_1^{\pm}$. Suppose that this singularity occurs at the point $(x_1,y_1)$. Since the map has corank $1$ at this point, we can apply Saji's criteria. Thus, there exist vector fields $\xi$ and $\eta$ with $\xi_0=(1,0)$ and $\eta_0=(0,1)$ in a neighborhood of $(x_1,y_1)$ such that $dF(\eta_0)=0$. It follows that $3y_1^2+a=0$ and $x_1+8y_1^7+2by_1+5cy_1^4=0$. Writing
\begin{center}
    $\varphi=$ \begin{tabular}{|ccc|}
$1$ & $0$ & $y_1$ \cr
$0$ & $3y_1^2+a$ &$x_1+8y_1^7+2by_1+5cy_1^4$ \cr
$0$ & $6y_1$ & $56y_1^6+2b+20cy_1^3$ \cr
\end{tabular},
\end{center}
it follows that we have a singularity worse than a cross-cap if $\xi\varphi=0$, that is, $6y_1=0$. Substituting $y_1=0$ into the relations above, we obtain $x_1=a=0$. On the other hand, it was shown in \cite{GibsonMarar} that there is no cross-cap (or worse) when $a\geq0$. Therefore, $H_3\not\rightarrow A_0S_1^{\pm}$.

Now suppose that $H_3\rightarrow [A_2]$.   
The bi-germ $[A_2]$ consists of two regular surfaces with $A_2$ tangency. Thus, consider $(x_1,y_1)\neq(x_2,y_2)$ such that $F(x_1,y_1,{\bf u})=F(x_2,y_2,{\bf u})$. Hence, $x_1=x_2=x$. Observe that $F_x=(1,0,y)$, $F_y=(0,3y^2+a,x+8y^7+2by+5cy^4)$, $F_{xx}=(0,0,0)$, $F_{yy}=(0,6y,56y^6+2b+20cy^3)$ and $F_{xy}=(0,0,1)$. 

The $A_2$ tangency condition is described in Proposition \ref{cond [A2]} and guarantees that $3y_1^2+a=0$. Moreover, the equation of the tangent plane at the points $(x_1,y_1)$ and $(x_2,y_2)$ is given by 
$$
y(3y^2+a)X+(x+8y^7+2by+5cy^4)Y-(3y^2+a)Z=0
$$
The fact that the tangent planes at $(x_1,y_1)$ and $(x_2,y_2)$ are equal gives
$$
(3y_2^2+a)(x+8y_1^7+2by_1+5cy_1^4)=0.
$$
There are two possibilities. If $(3y_2^2+a)=0$ and $3y_1^2+a=0$, then $y_1=-y_2$. Thus, from the equality $F(x_1,y_1)=F(x_2,y_2)$, we obtain $y_1^3+ay_1=y_2^3+ay_2$. Substituting $a$ and $y_1$ into the previous relation, we obtain $y_2=0$. This is a contradiction.
If $(x+8y_1^7+2by_1+5cy_1^4)=0$, then $(x_1,y_1)$ is a singular point, contradicting the fact that it is a regular point.

Suppose that $H_3\rightarrow A_0S_0|A^{\pm}_1$. Let $(x_1,y_1)$ be the regular point and $(x_2,y_2)$ the singular point where the cross-cap occurs, such that $F(x_1,y_1,{\bf u})=F(x_2,y_2,{\bf u})$. At the point $(x_2,y_2)$, Saji's criteria guarantees that $3y_2^2+a=0$, $x_2+8y_2^7+2by_2+5cy_2^4=0$ and $y_2\neq0$. From the relation $F(x_1,y_1,{\bf u})=F(x_2,y_2,{\bf u})$, we obtain $x_1=x_2=-(8y_2^7+2by_2+5cy_2^4)$, $a=-3y_2^2$, $y_1=-2y_2$ and $b=-31y^6_2+2cy^3_2$. On the other hand, the regular surface is tangent to the cross-cap, so the tangent plane to the regular surface at the point $(x_1,y_1)$, described in the previous case, contains the tangent vector to the cross-cap, that is, $(-3y_1^3-ay_1)-y_2(3y_1^2+a)=0.$ Substituting $a$ and $y_1$, we obtain $y_2=0$, which cannot occur.

Suppose that $H_3\rightarrow A_0^3|A_2$. Take $(x_i,y_i)$, with $i=1,2,3$, to be distinct regular points such that $F(x_i,y_i,{\bf u})=F(x_j,y_j,{\bf u})$ for $i\neq j$.  
From the equality $F(x_i,y_i,{\bf u})=F(x_j,y_j,{\bf u})$, we obtain $x_1=x_2=x_3$, $(y_i-y_j)(y_i^2+y_iy_j+a+y_j^2)=0$ and $(x_iy_i+y_i^8+by_i^2+cy_i^5-x_jy_j+y_j^8+by_j^2+cy_j^5)=0$ with $i\neq j$. 
Since $x_1=x_2=x_3$, then $y_1\neq y_2\neq y_3$. Thus, subtracting $(y_1^2+y_1y_2+a+y_2^2)=0$ and $(y_2^2+y_2y_3+a+y_3^2)=0$ gives $-(y_1-y_3)(y_1+y_2+y_3)=0$. Hence, $y_1=-(y_2+y_3)$. 
Moreover, the equation $(x_iy_i+y_i^8+by_i^2+cy_i^5-x_jy_j+y_j^8+by_j^2+cy_j^5)=0$ with $i=1$ and $j=2$ gives $x_2=p_1(y_1,y_2,c)$, where $p_1$ is a polynomial of degree 7 in the variables $y_1,y_2$ and $c$. 
For $i=2$ and $j=3$, we obtain $x_2=p_2(y_2,y_3,c)$, where $p_1$ is a polynomial of degree 7 in the variables $y_2,y_3$ and $c$. Subtracting the expressions for $x_2$ and substituting $y_1=-(y_2+y_3)$, we find $b=p_3(y_2,y_3)$, where $p_3$ is a polynomial of degree 6 in the variables $y_2$ and $y_3$.

On the other hand, the general equation of the tangent planes at the points $(x_i,y_i)$ in the coordinates $(X,Y,Z)$ is given by 
$$
-y_i(3y_i^3+a)X+(-x_i-8y^7_i-2by_i-5cy_i^4)Y+(3y_i^3+a)Z=0
$$
Since the tangent planes at the points $(x_i,y_i)$ are transverse to each other, substituting the expressions found above into the condition that the planes are transverse gives 
$$
-(y_2+2y_3)(2y_2+y_3)(y_2-y_3)^2(1+p_4(y_2,y_3,c))=0,
$$  
where $p_4$ is a polynomial, that is, $(y_2+2y_3)(2y_2+y_3)=0$. Thus, if $y_2=-2y_3$, then substituting into the expression $y_1+y_2+y_3=0$ gives $y_1=y_3$, which cannot occur. Therefore, $H_3\not\rightarrow A_0^3|A_2$.

The case $H_3\not\rightarrow(A_0^2|A_1^{\pm})A_0$ is analogous to $H_3\not\rightarrow A_0^3|A_2$. \qed
\end{proof}


\begin{proposition}\label{cond(A0S0)2}
The class of singularity $H_3$ is adjacent to the bi-germ $(A_0S_0)_2$.
\end{proposition}
\begin{proof} 
Consider the deformation of $H_3$ as in Proposition \ref{H_3not}. The conditions for the bi-germ $(A_0S_0)_2$ are given by:
\begin{itemize}
  \item[(i)] $F(x_1,y_1,{\bf u})=F(x_2,y_2,{\bf u})$ for $(x_1,y_1)\neq(x_2,y_2)$;
  \item[(ii)] The point $(x_2,y_2)$ is a cross-cap;
  \item[(iii)] The limiting tangent vector of the double point curve of the cross-cap is contained in the tangent space to the image of $F_{\bf u}$ at $(x_1,y_1)$.
\end{itemize}

By Saji's criteria, $F$ has a cross-cap singularity at $(x_2,y_2)$ if
$a=-3y^2_2$, $x_2=-(8y_2^7+2by_2+5c y_2^4)$ and $y_2\neq0$. Substituting $a$ and $x_2$ into the equality $F(x_1,y_1,{\bf u})=F(x_2,y_2,{\bf u})$, we obtain $x_1=x_2$ and
$-(2y_2+y_1)(y_2-y_1)^2=0$,
which implies that $y_1=-2y_2$. Substituting the relations obtained above into
$$
x_1y_1+y_1^8+by_1^2+cy_1^5-(x_2y_2+y_2^8+by_2^2+cy_2^5)=0,
$$
we obtain $-9y_2^2(31y_2^6-2c y_2^3+b)=0$. Since $y_2\neq0$, it follows that
$b=-(31y_2^6-2c y_2^3)$.

Finally, we need to determine the limiting tangent vector of the double point curve of the cross-cap. To this end, we first perform the translations $x\mapsto x+x_2$ and $y\mapsto y+y_2$ in order to move the singularity $(x_2,y_2)$ to the origin. Thus, in the new coordinate system, $F$ has the form
$$
(x+x_2,(y+y_2)^3+a(y+y_2),(x+x_2)(y+y_2)+(y+y_2)^8+b(y+y_2)^2+c(y+y_2)^5)
$$
and we denote it by $\widetilde{F}(x,y,{\bf u})$.
Searching for the double point curve, we consider
$\widetilde{F}(\bar{x}_1,\bar{y}_1,{\bf u})=\widetilde{F}(\bar{x}_2,\bar{y}_2,{\bf u})$
with $(\bar{x}_1,\bar{y}_1)\neq(\bar{x}_2,\bar{y}_2)$. This holds if and only if
\begin{enumerate}
  \item $\bar{x}_1=\bar{x}_2$;
  \item $(\bar{y}_1+y_2)^3+a(\bar{y}_1+y_2)=(\bar{y}_2+y_2)^3+a(\bar{y}_2+y_2)$;
  \item $(\bar{x}_1+x_2)(\bar{y}_1+y_2)+(\bar{y}_1+y_2)^8+b(\bar{y}_1+y_2)^2+c(\bar{y}_1+y_2)^5=(\bar{x}_2+x_2)(\bar{y}_2+y_2)+(\bar{y}_2+y_2)^8+b(\bar{y}_2+y_2)^2+c(\bar{y}_2+y_2)^5.$
\end{enumerate}

Substituting $a$ and $x_2$ into (2), we obtain
$$
(\bar{y}_1-\bar{y}_2)(\bar{y}_1^2+3y_2\bar{y}_1+\bar{y}_1\bar{y}_2+3y_2\bar{y}_2+\bar{y}^2_2)=0.
$$
Since $(\bar{y}_1-\bar{y}_2)\neq0$ and $y_2\neq0$, the Implicit Function Theorem guarantees that $\bar{y}_1$ can be written as a function of $\bar{y}_2$, namely,
$$
\bar{y}_1=-\bar{y}_2-\frac{1}{3y_2}\bar{y}_2^2-\frac{1}{9y_2^2}\bar{y}_2^3+\cdots.
$$
Equation (3) reduces to
$(-\bar{y}_2+\bar{y}_1)(\bar{x}_1+f(\bar{y}_1,\bar{y}_2,b,c))=0$,
where $f$ is a function depending on $\bar{y}_1,\;\bar{y}_2,\;b$ and $c$.
Substituting the expressions for $b$ and $\bar{y}_1$ obtained above, we can write $\bar{x}_1$ as a function of $\bar{y}_2$, obtaining
$$
\bar{x}_1=(-6c y^2_2-57y^5_2)\bar{y}^2_2+(-2c y_2-19y^4_2)\bar{y}^3_2+\cdots.
$$
Thus, the double point curve is parametrized by
\begin{eqnarray*}
\alpha(\bar{y}_2)&=&\Big(-y_2(8 y^6_2+2b+5c y^3_2)-y_2(6c y_2+57y^4_2)\bar{y}^2_2+\cdots,\\
&&-2y_2^3+3y_2\bar{y}_2^2+\cdots,6y^5_2(-4y^3_2+c)+6y^3_2(-10y^3_2+c)\bar{y}^2_2+\cdots\Big).
\end{eqnarray*}
Hence, the limiting tangent vector at the origin is given by
$$
\alpha''(0)=(-6y^2_2(2c +19y^3_2),6y_2,12y^3_2(-10y^3_2+c)).
$$

On the other hand, the equation of the tangent plane at the point $(x_1,y_1)$ is
$$
y_1(3y_1^2+a)X+(x_1+8y_1^7+2by_1+5cy_1^4)Y-(3y_1^2+a)Z=0.
$$
Thus, requiring $\alpha''(0)$ to belong to this plane and substituting the expressions for $a$, $b$, $x_1$ and $y_1=-y_2$, we obtain
$$
486y_2^5(-4y_2^3+c)=0.
$$
Therefore, we obtain $c=4y_2^3$.
Substituting $c$ successively into $a$, $b$ and $x_1$, we obtain the curve
$$
\beta(y_2)=\Big\{\left(18y_2^7,-2y_2,-3y^2_2,-23y_2^6,4y_2^3\right)\;|\;y_2\in\mathbb R\Big\}
$$
in $(\mathbb{R}^2\times\mathbb{R}^3,0)$. Projecting by
$\pi:(\mathbb{R}^2\times\mathbb{R}^3,0)\to(\mathbb{R}^2,0)$
yields the curve
$\pi(\beta(y))=(18y_2^7,-2y_2)$
formed by the bi-germ $(A_0S_0)_2$ in a neighborhood of the singularity $H_3$ in $(\mathbb{R}^2,0)$. Therefore, we obtain the $\mathcal A$-adjacency. \qed
\end{proof}

\medskip

Figure \ref{fig:adjacencias_codim3_codim2} summarizes the results obtained above, along with the adjacencies of the $P_3(c)$-singularity established in \cite{Deolindo}.

\begin{figure}[ht]
\centering
\begin{tikzpicture}[
    >=stealth,
    every node/.style={font=\footnotesize},
    arrow/.style={->, thick}
]

\node (dots1) at (0,2.8) {$\cdots$};
\node (H3) at (2,2.8) {$H_3$};
\node (A0S02) at (5.2,2.0) {$(A_0S_0)_2$};

\draw[arrow] (dots1) -- (H3);
\draw[arrow] (H3) -- (A0S02);

\node (dots2) at (0,0.8) {$\cdots$};
\node (P3) at (2.2,0.8) {$P_3(c)$};
\node (A0S1) at (5.2,0.8) {$A_0S_1^{\pm}$};
\node (A0S0A1) at (5.4,-0.3) {$A_0S_0|A_1^{\pm}$};

\draw[arrow] (dots2) -- (P3);
\draw[arrow] (P3) -- (A0S02);
\draw[arrow] (P3) -- (A0S1);
\draw[arrow] (P3) -- (A0S0A1);

\node (dots3) at (0,-1.6) {$\cdots$};
\node (B3) at (2.2,-1.6) {$B_3$};
\node (A2) at (5.0,-1.6) {$[A_2]$};

\draw[arrow] (dots3) -- (B3);
\draw[arrow] (B3) -- (A2);

\end{tikzpicture}
\caption{Adjacencies of codimension $3$ mono-germ $(\mathbb{R}^2,0)\to(\mathbb{R}^3,0)$ to codimension $2$ bi-germs $(\mathbb{R}^2,0)\to(\mathbb{R}^3,0)$.}
\label{fig:adjacencias_codim3_codim2}
\end{figure}
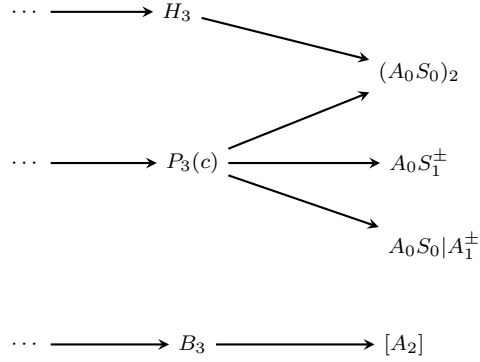

\section{Some local and multi-local robust features}\label{robust}

Having established in Section \ref{adjacency} the adjacencies from codimension $3$ local singularities to codimension $2$ multi-local singularities, we now determine the loci of codimension $2$ singularities arising in the families of height functions and orthogonal projections on surfaces in $\mathbb R^4$, and derive their corresponding geometric characterizations.
We divide the analysis into two cases, treating separately the characterizations associated with the family of height functions and those associated with the family of orthogonal projections.

In both cases, let $M$ be a regular surface and let $p\in M$. Locally around $p$, we write $M$ in Monge form: 
%
%
\[
\varphi(x,y)=(x,y,f^1(x,y),f^2(x,y)),
\]
where
\begin{equation}\label{Monge}  
\begin{array}{l} 
f^1(x,y)=\displaystyle\sum^{2}_{i=0}a_{2i}x^{2-i}y^{i}+
 \sum^{3}_{i=0}a_{3i}x^{3-i}y^{i}+
 \sum^{4}_{i=0}a_{4i}x^{4-i}y^{i}+
\cdots,\\
 f^2(x,y)=\displaystyle\sum^{2}_{i=0}b_{2i}x^{2-i}y^{i}+
 \sum^{3}_{i=0}b_{3i}x^{3-i}y^{i}+
 \sum^{4}_{i=0}b_{4i}x^{4-i}y^{i}+
\cdots.
\end{array} 
\end{equation}
Of course, we can simplify the $2$-jet of $(f^1,f^2)$ further and reduce it to one of the normal forms for binary pairs described previously in Table \ref{Gibson}.

\subsection{Families of height functions}

\

\medskip

We may consider a modified family of height functions associated with the binormal direction $v_0=(0,0,1,0)$, given locally by
\begin{equation}\label{ModHeight}
h(x,y,{\bf v})=ux+vy+f^1(x,y)+wf^2(x,y),
\end{equation}
where ${\bf v}=(u,v,1,w)$, and 
$h_{v_0}=h(x,y,0,0,1,0)=f^1(x,y)$. 
%
%
Follows from the results of Montaldi in
\cite{Montaldi,Montaldi2} that
%
for a generic surface, the singularities that may occur in the family of height functions consist of the local singularities $A_k$, $1\leq k\leq4$, and $D_4$, together with the multi-local singularities
$2A_1$, $3A_1$, $4A_1$, $A_1A_2$, $A_1A_3$, $2A_1A_2$, and $A_2A_2$.
For $j^2(f^1,f^2)=(a_{20}x^2,b_{20}x^2+a_{21}xy+a_{22}y^2)$, the following result can be found in \cite{BruceTari}.

\begin{proposition}[\cite{BruceTari}]\label{heigthMonge}
With the notation above, the conditions for the generic local singularities of the height function $h_{v_0}$ are given by
$$
\begin{array}{ccl}
  A_2: & a_{20}\neq0, & a_{33}\neq0,  \\
  A_3: & a_{20}\neq0, & a_{33}=0,\;\; 4a_{20}a_{44}-a_{32}^2\neq0,  \\
  A_4: & a_{20}\neq0, & a_{33}=0,\;\; 4a_{20}a_{44}-a_{32}^2=0,\;\; 4a_{20}^2a_{55}-2a_{20}a_{32}a_{43}+a_{31}a_{32}^2\neq0, \\
  D_4: & a_{20}=0, & a_{33}x^3+a_{32}x^2y+a_{31}xy^2+a_{30}y^3\;\;\; \mbox{is non-degenerate.}
  \end{array}
$$

These singularities are versally unfolded by the family of height functions $h$ if and only if
$$
\begin{array}{cl}
  A_2: & \mbox{always,}  \\
  A_3: & b_{22}\neq0 \mbox{ or } a_{32}\neq0,  \\
  A_4: & b_{22}(a_{20}a_{43}-a_{31}a_{32})-a_{32}(a_{20}b_{33}-b_{21}a_{32})\neq0, \\
  D_4: & 3a_{30}(b_{22}a_{32}-3b_{21}a_{33})-a_{31}(b_{22}a_{31}-b_{21}a_{32})+b_{20}(3a_{31}a_{33}-a^2_{32})\neq0.
  \end{array}
$$
\end{proposition}

As the family of height functions $h$ depends on 3 parameters, for a generic immersed surface in $\mathbb R^4$, the singularities of $\mathcal{R}^+_e$-codimension $2$ form smooth curves on the surface, while those of $\mathcal{R}^+_e$-codimension $3$ occur as isolated points along these curves. In particular, since $A_4$ is an $\mathcal{R}^+_e$-codimension $3$ singularity and is adjacent to the $\mathcal{R}^+_e$-codimension $2$ singularities $A_3$ and $A_1A_2$ (Section \ref{adjacency}), we obtain the following result:




\begin{theorem}\label{tang A1A2} Let $p$ be an $A_4$-singularity of the height function $h_{v_0}$ on a surface $M$ in $\mathbb R^4$, which we call the {\normalfont $A_4$-point}. Then there exists a neighborhood of $A_4$-point such that:
\begin{enumerate}
\item[(1)] 
the $A_3$-curve is regular and  can be parametrized locally by
$$
x = \frac{1}{2}\frac{(4a_{31}a_{32}^2-9a_{32}a_{43}+20a_{55}-b_{33}a_{32}^2)}{(a_{31}a_{32}+a_{32}b_{33}-a_{43})}y^2+O(3),
$$
\item[(2)] 
the $A_1A_2$-curve is regular and can be parametrized locally by
$$
x=\frac{1}{16}\frac{(7a_{31}a_{32}^2-22a_{32}a_{43}+60a_{55}-8a_{32}^2b_{33})}{(a_{32}a_{31}+b_{33}a_{32}-a_{43})}y^2+O(3)
$$
\item[(3)] the $A_3$-curve and the $A_1A_2$-curve are generically tangent at the $A_4$-point and have contact of order $2$.
\end{enumerate}
\end{theorem}

\begin{proof}
Consider the modified family of height functions $h$ 
given as in \eqref{ModHeight}
Assume that $h$ is a versal deformation of an $A_4$-singularity, that is, $a_{33}=0$, $4a_{44}-a_{32}^2=0$, $4a_{55}-2a_{32}a_{43}+a_{31}a_{32}^2\neq0$ and $a_{43}-a_{31}a_{32}-a_{32}b_{33}\neq0$. 

$(1)$ In order to have an $A_3$-singularity, it is necessary that the quadratic part of $h$ divides the cubic part at the singular point $(x,y)$, that is,
\begin{itemize}
\item[(i)] $h_x(x,y,{\bf v})=h_y(x,y,{\bf v})=0$,
\item[(ii)] $(h_{xy}^2-h_{xx}h_{xy})(x,y,{\bf v})=0$,
\item[(iii)] $\left(h_{xy}^3h_{yyy}+3h_{xy}^2h_{yy}h_{xyy}-3h_{xy}h_{yy}^2h_{xxy}-h_{yy}^3h_{xxx}\right)(x,y,{\bf v})=0.$
\end{itemize}
From condition (i), we obtain
$$
\begin{array}{l}
 u=-2x-a_{32}y^2-3a_{30}x^2-2a_{31}xy+O(3)\\
 v=-a_{31}x^2-2a_{32}xy-2yw+O(3)
\end{array}
$$
Substituting these expressions into (ii), we obtain
$$
w= -a_{32}x+(a_{31}^2+a_{32}b_{32}-a_{42})x^2+(2a_{32}a_{31}+3b_{33}a_{32}-3a_{43})xy-\frac12 a_{32}^2y^2+O(3).
$$
Finally, substituting the expressions for $u,v$ and $w$ into (iii), it follows that\\
$2(a_{31}a_{32}+a_{32}b_{33}-a_{43})x+(4a_{31}a_{32}^2-9a_{32}a_{43}+20a_{55}-b_{33}a_{32}^2)y^2$+\\
$(-6b_{33}a_{43}+6a_{32}b_{33}^2-8a_{32}b_{44}-8a_{32}a_{42}+4a_{32}^2b_{32}+8a_{54}+6a_{30}a_{32}^2-2a_{31}a_{32}b_{33})yx$\\
$+(2b_{33}a_{32}b_{32}-18a_{30}a_{32}b_{33}-12a_{32}a_{31}a_{30}-3a_{41}a_{32}+2a_{32}^2b_{31}-4a_{31}a_{42}-2a_{32}b_{43}+2a_{53}+2a_{31}a_{32}b_{32}-2b_{33}a_{42}+18a_{30}a_{43}+2b_{33}a_{31}^2+2a_{31}^3)x^2+O(3)=0$.\\
Using the Implicit Function Theorem, we can write $x$ as a function of $y$, and we obtain the desired parametrization. 

$(2)$ To have an $A_1A_2$-singularity, we need to find distinct points $(x_1,y_1)$ and $(x_2,y_2)$ such that
$
h(x_1,y_1,{\bf v})=h(x_2,y_2,{\bf v}),
$
that is, a double point, with an $A_1$-singularity at $(x_1,y_1)$ and an $A_2$-singularity at $(x_2,y_2)$. Thus, the following equations need to be satisfied:
\begin{itemize}
\item[(i)] $h(x_1,y_1,{\bf v})=h(x_2,y_2,{\bf v})$,
\item[(ii)] $h_x(x_i,y_i,{\bf v})=h_y(x_i,y_i,{\bf v})=0$ for $i=1,2$,
\item[(iii)] $(h_{xy}^2-h_{xx}h_{xy})(x_2,y_2,{\bf v})=0.$
\end{itemize}

Condition (iii) guarantees that
$$
w=-a_{32}x_2+(-a_{42}+a_{32}b_{32}+a_{31}^2)x_2^2-\frac12a_{32}^2y_2^2+(2a_{31}a_{32}-3a_{43}+3b_{33}a_{32})x_2y_2+O(3)
$$
Substituting this into $h_y(x_1,y_1,{\bf v})=h_x(x_1,y_1,{\bf v})=0$, we obtain
$$
\begin{array}{l}
v= -2a_{32}x_1y_1+2a_{32}x_2y_1-a_{31}x_1^2+O(3),\\
u= -2x_1-a_{32}y_1^2-2a_{31}x_1y_1-3a_{30}x_1^2+O(3).
\end{array}
$$
Substituting $u,\;v$ and $w$ into $h_x(x_2,y_2,{\bf v})=0$ and using the Implicit Function Theorem, we can write $x_1$ as a function of $x_2$ and $y_2$, obtaining
$$
x_1 = x_2+a_{31}x_2y_2+(1/2)a_{32}y_2^2-(1/2)a_{32}y_1^2-a_{31}x_2y_1+O(3).$$
Substituting this into $h_y(x_2,y_2,{\bf v})=0$, we obtain\\
$
\displaystyle\frac{1}{4}(y_1-y_2)^2\Big((12a_{43}-12a_{32}a_{31}-12b_{33}a_{32})x_2+(60a_{55}+9a_{31}a_{32}^2-24a_{32}a_{43}-6a_{32}^2b_{33})y_2^2$\\$+(20a_{55}+5a_{31}a_{32}^2-10a_{32}a_{43})y_1^2+(10 a_{31} a_{32}^2+40 a_{55}-20 a_{32} a_{43}) y_2 y_{1}+O(3)\Big)=0.
$\\
Since $h$ is a versal unfolding of the $A_4$-singularity and $a_{43}-a_{32}a_{31}-b_{33}a_{32}\neq0$, it follows from the Implicit Function Theorem that
$$
\begin{array}{l}
\displaystyle x_2 = \frac{5}{6}\frac{(a_{31}a_{32}^2-2a_{32}a_{43}+4a_{55})}{(a_{31}a_{32}+a_{32}b_{33}-a_{43})}y_1y_2
+\frac{5}{12}\frac{(a_{31}a_{32}^2-2a_{32}a_{43}+4a_{55})}{(a_{31}a_{32}+a_{32}b_{33}-a_{43})}y_1^2+\\
\displaystyle \;\;\;\;\;\;\;\;\;\;\;\;+\frac{1}{4}\frac{(20a_{55}-8a_{32}a_{43}+3a_{31}a_{32}^2-2a_{32}^2b_{33})}{(a_{32}a_{31}+b_{33}a_{32}-a_{43})}y_2^2+O(3).
\end{array}
$$
Finally, substituting all the obtained equations into equation (i), we obtain
$$
-\frac{1}{12}(2y_1+3y_2)(y_1-y_2)^4(a_{31}a_{32}^2-2a_{32}a_{43}+4a_{55})=0.
$$
Since $y_1-y_2\neq0$ and $a_{31}a_{32}^2-2a_{32}a_{43}+4a_{55}\neq0$, we obtain $y_1 = -(3/2)y_2$.

Thus, the set of points $(p,{\bf v})$ in $M\times S^3$ such that $h$ has an $A_1A_2$-singularity at $p$ is a smooth curve. Projecting this curve onto $M$, we obtain a curve parametrized by
$$
 \displaystyle x_2=\frac{1}{36}\frac{(7a_{31}a_{32}^2-22a_{32}a_{43}+60a_{55}-8a_{32}^2b_{33})}{(a_{32}a_{31}+b_{33}a_{32}-a_{43})}y_1^2+O(3) \text{ and $ y_2 = -(2/3)y_1$}
$$

 $(3)$ The proof follows using $(1)$ and $(2)$.
\qed
\end{proof}


\begin{remark}{\normalfont
\begin{enumerate}
\item[(i)] The $A_3$-curve and the $A_1A_2$-curve have order of contact equal to two. Indeed, let $c_1$ and $c_2$ denote the quadratic coefficients of these curves, respectively. Then
$$
c_1-c_2=-\frac {25}{16}\frac {(a_{31} a_{32}^{2}-2 a_{32}a_{43}+4a_{55})}{ a_{32} a_{31}+ b_{33} a_{32}-a_{43}}
$$
is nonzero by the existence and versality conditions for the $A_4$-singularity. Moreover, all possible relative positions of the $A_1A_2$-curve and the $A_3$-curve at an $A_4$-singularity can occur.

\item[(ii)] Although the $A_3$-curve was determined and studied in \cite{BruceTari,dmm}, it was not treated from the perspective of the study of multi-local singularities. In \cite{dmm}, the authors determined that the $A_3$-curve is smooth and tangent to the parabolic set. In \cite{BruceTari}, they describe all possible transitions of the parabolic set and the $A_3$-curve that may occur in generic one-parameter families of surfaces.
\end{enumerate}}
\end{remark}


\subsection{The family of orthogonal projections}

\

\medskip


The method described in the case of the family of height functions can also be used to determine the robust curves came from local and multi-local singularities of the family of orthogonal projections on the surface $M$. 

Consider the family of orthogonal projections $P:M\times S^3\rightarrow TS^3$ given as in \eqref{OrtProj} with
$
P(p,{\bf u})=({\bf u},p-\langle p,{\bf u}\rangle {\bf u})
$
and 
$P_{\bf u}(p)=p-\langle p,{\bf u}\rangle {\bf u}.$
For fixed ${\bf u}$, the projection can be viewed locally around $p\in M$ as the germ
$
P_{\bf u}: (\mathbb{R}^2,0)\rightarrow(\mathbb{R}^3,0).
$
%
%
The generic $\mathcal A$-singularities of $P_{\bf u}$
are those that are listed in Table \ref{mond} and a list of bi-germs and tri-germs is in Table \ref{bigerm}. Recall that the map $P_{\bf u}$ has a singularity worse than a cross-cap whenever ${\bf u}$ is an asymptotic direction. 

We denote by $(X,Y,Z,W)$ the coordinates in $\mathbb R^4$ and fix ${\bf u}=(0,1,0,0)$ as an asymptotic direction. We parametrize the neighboring directions by $(u,1,v,w)$. Instead of projecting orthogonally onto the plane $(u,1,v,w)^\perp$, we project onto the fixed plane $(X,Z,W)$.
The modified family of projections is given by
$$
\begin{array}{cccl}
                 P: & (\mathbb{R}^2\times \mathbb{R}^3,0) & \rightarrow & (\mathbb{R}^3,0) \\
                  & ((x,y),(u,v,w)) & \mapsto & P=(x-uy,f^1(x,y)-vy,f^2(x,y)-wy),
               \end{array}
$$
with $P_\textbf{u}(x,y)=(x,f^1(x,y),f^2(x,y))$.

The singularities $S_k$, $B_k$, and $C_3$ of $P_\textbf{u}$ occur at hyperbolic points. In these cases, we may take the Monge form as in \eqref{Monge} with $j^2(f^1,f^2)=(x^2,y^2)$. 


\begin{proposition}\label{ClassMonge1}
With the above notation, the conditions for the projection $P_{\bf u}$ to have the following generic $\mathcal{A}$-singularities from Table \ref{mond} in terms of the coefficients of the
Monge form $(f^1,f^2)$ at $p$ are:
$$
\begin{array}{rl}
  S_1: & a_{31}\neq 0,\; a_{33} \neq 0;\\
   S_2: & a_{31} = 0,\; a_{33}\neq 0,\; a_{41}-a_{32}b_{31}\neq 0; \\
  S_3: & a_{31} = 0,\; a_{33}\neq 0,\; a_{41}-a_{32}b_{31}=0;\\
       & 3a_{33}b_{31}^2-4a_{42}b_{31}+4a_{51}-4a_{32}(-b_{31}b_{32}+b_{41})\neq0;\\
  B_2:& a_{31}\neq 0,\; a_{33} = 0,\; 4a_{32}^2b^2_{33}+8a_{31}a_{44}b_{33}-2a_{32}a_{43}b_{33}-4a_{31}a_{55}+a^2_{43}\neq 0; \\
  B_3: &  a_{31}\neq 0,\; a_{33} = 0,\; 4a_{32}^2b^2_{33}+8a_{31}a_{44}b_{33}-2a_{32}a_{43}b_{33}-4a_{31}a_{55}+a^2_{43}= 0; \\
       &  \xi\neq0\\
  C_3: &  a_{31} = 0,\; a_{33}= 0,\; -a_{32}b_{31}+a_{41}\neq 0,\;-a_{32}b_{33}+a_{43}\neq0;
\end{array}
$$
where the expression $\xi$ depends on the coefficients of the $7$-jet of $f^1$ and $f^2$.
Moreover, the singularities $S_1$, $B_2$, $S_2$, and $C_3$ of $P_{\bf u}$ are versally unfolded by the family of orthogonal projections $P$ if and only if
$$
\begin{array}{rl}
  S_1: & \text{always};\\
  S_2: & -4a_{30}a_{32}-2a_{32}b_{32}-3a_{33}b_{31}+2a_{42}\neq0; \\
   B_2: & a_{32}^2-4a_{44}\neq 0 \;\; \text{and} \;\; a_{32}a_{31}+b_{33}a_{32}-a_{43}\neq0; \\
  C_3: & \begin{array}{l}(b_{31}a_{32}-a_{41})(a_{32}^2-4a_{44})-\\(-\frac14a_{43}+\frac14b_{33}a_{32})(a_{30}a_{32}+\frac23b_{32}a_{32}-\frac23a_{42})\neq0.\end{array}
\end{array}
$$
\end{proposition}
\begin{proof}
The proof follows the same arguments as those used in Theorem~6.7 and Theorem~6.8 in \cite{IFRT-book}, and is therefore omitted. The computations were carried out with the aid of Maple.\qed
\end{proof}

\

Since the singularities $H_2$, $H_3$, and $P_3(c)$ of $P_{\bf u}$ occur at parabolic points, we take the Monge form as in \eqref{Monge} with 
$j^2(f^1,f^2)=(x^2,xy).$

\begin{proposition}\label{CondPar}
With the notation above, the conditions for $P_{\bf u}$ to have the following generic $\mathcal A$-singularities listed in Table \ref{mond} in terms of the coefficients of the
Monge form $(f^1,f^2)$ at $p$ are:
$$
\begin{array}{cl}
  H_2: & a_{33}\neq 0,\; (-b_{32}b_{44}+b_{55})a_{33}^2+((a_{32}b_{44}+a_{44}b_{32}-a_{55})b_{33}-a_{44}b_{44})a_{33}- \\
  & \qquad\qquad\qquad\qquad-a_{32}a_{44}b_{33}^2+a_{44}^2b_{33}\neq 0;\\
  H_3: & a_{33}\neq 0,\; (-b_{32}b_{44}+b_{55})a_{33}^2+((a_{32}b_{44}+a_{44}b_{32}-a_{55})b_{33}-a_{44}b_{44})a_{33}- \\
  & \qquad\qquad\qquad\qquad-a_{32}a_{44}b_{33}^2+a_{44}^2b_{33}=0,\;\zeta\neq0;\\
  P_3(c): & a_{33}=0,\;
  b_{33}a_{32}a_{44}\neq0,\;
  \displaystyle\frac{a_{44}}{b_{33}a_{32}}
  \neq0,\frac12,1,\frac32.
\end{array}
$$
where $\zeta$ is an expression depending on the coefficients of the $8$-jet of $f^1$ and $f^2$.

\end{proposition}

\begin{proof}
The proof follows the same arguments as those of Proposition \ref{ClassMonge1}. 
\qed
\end{proof}

\begin{remark}
{\normalfont
We do not state the versality conditions for the singularities $S_3$, $B_3$, $H_2$, and $H_3$ in Propositions \ref{ClassMonge1} and \ref{CondPar}, since these conditions depend on at least the $5$-jet of $(f^1,f^2)$ and are not needed for the results developed in this work.
}
\end{remark}

The conditions on the coefficients of the Monge form established in the previous propositions will be essential for determining the parametrization of the local and multilocal curves, as well as for describing their features on the surface $M$. We divide the discussion into two cases: the local features and the multilocal features.

\subsubsection{Local robust features}

\begin{proposition}\label{B2curve}
Let $p$ be a hyperbolic point of $M$ satisfying the conditions of Proposition \ref{ClassMonge1}. Then there exists a neighbourhood of $p$ such that:
\begin{enumerate}
\item[(1)] 
the $B_2$-curve is regular and can be parametrized locally by
$$
(-a_{43}+a_{32}a_{31}+a_{32}b_{33})x-(a_{32}^2-4a_{44})y+O(2)=0.
$$

\item[(2)] if $p$ is a $B_3$-singularity of $P_{\mathbf u}$, then the $B_2$-curve can be parametrized locally by
$$
y=-\frac{a_{31}a_{32}+a_{32}b_{33}-a_{43}}{a_{32}^2-4a_{44}}x+Ax^2+O(3),
$$
where the expression $A$ depends on the coefficients of the $5$-jet of $f^1$ and $f^2$.


\item[(3)] if $p$ is a $C_3$-singularity of $P_{\mathbf u}$, then the $B_2$-curve can be parametrized locally by
$$
(-a_{43}+a_{32}b_{33})x-(a_{32}^2-4a_{44})y+O(2)=0.
$$
\end{enumerate}
\end{proposition}

\begin{proof}
 At a hyperbolic point, we may assume 
the Monge form as in \eqref{Monge} with $j^2(f^1,f^2)=(x^2,y^2)$. 

\noindent\noindent(1) As the origin is a $B_2$-singularity we take the projection $P_{\bf u}$ as in Proposition \ref{ClassMonge1}. So 
$a_{33}=0$, $a_{31}\neq0$, and
$
4a_{32}^2b_{33}^2
+8a_{31}a_{44}b_{33}
-2a_{32}a_{43}b_{33}
-4a_{31}a_{55}
+a_{43}^2\neq0.
$
To simplify the calculation, we use the change  of coordinates 
$
(x,y)\mapsto(\bar{x}+u\bar{y},\bar{y}).
$
With a slight abuse of notation, we may therefore write the projection as
\begin{equation}\label{ModP}
P
=
({x},
f^1({x}+u{y},{y})-v{y},
f^2({x}+u{y},{y})-w{y}).
\end{equation}
Write
$
(g^1,g^2,g^3)
=
(x,f^1(x+uy,y)-vy,f^2(x+uy,y)-wy).
$
If $(x,y)$ is a singular point of $P$, then
$$
\begin{array}{l}
w=2y+b_{31}x^2+2b_{32}xy+3b_{33}y^2+\cdots,\\
v=a_{31}x^2+2a_{32}xy+2ux+\cdots.
\end{array}
$$

Since $B_2$-singularity is more degenerate than a cross-cap, Saji's criteria together with the Implicit Function Theorem yields
\begin{eqnarray*}
u
&=&
-a_{32}y-a_{31}x
+\left(-\frac32a_{43}+2a_{32}a_{31}\right)y^2
+\left(3a_{31}a_{30}-\frac32a_{41}+b_{31}a_{32}\right)x^2
\\
&&
+\left(-2a_{42}+2a_{31}^2+3a_{32}a_{30}+a_{32}b_{32}\right)xy
+\cdots.
\end{eqnarray*}
A $B_2$-singularity occurs at $(x,y)$ if
$g^2_{yyy}g^3_{yy}-g^2_{yy}g^3_{yyy}=0.$
Substituting the expressions for $u$, $v$, and $w$ into the above equation and reverting to the initial coordinate system gives the desired parametrization of the smooth $B_2$-curve
$$(-a_{43}+a_{32}a_{31}+a_{32}b_{33})x
+
(-4a_{44}+a_{32}^2)y
+ \Lambda y^2+\cdots=0,
$$
where 
{\footnotesize$$
\Lambda=\Big(
-a_{32}\left(-\frac32a_{43}+2a_{32}a_{31}\right)
-a_{31}a_{32}^2
-10a_{55}
+4a_{43}a_{32}
-(12a_{44}-3a_{32}^2)b_{33}
+(-2a_{32}^2+6a_{44})b_{33}
\Big).
$$}

\medskip

(2) At a $B_3$-singularity, we have the same conditions as in statement (1), together with
$
4a_{32}^2b_{33}^2
+
8a_{31}a_{44}b_{33}
-
2a_{32}a_{43}b_{33}
-
4a_{31}a_{55}
+
a_{43}^2
=0.
$
In this case, the Implicit Function Theorem gives the parametrization
$$
y=
-\frac{a_{31}a_{32}+a_{32}b_{33}-a_{43}}
{a_{32}^2-4a_{44}}
x
+
Ax^2+\cdots,
$$
where $A$ is a lengthy expression depending on the coefficients of the $5$-jet of $(f^1, f^2)$.

\medskip






(3) Finally, in the case of a $C_3$-singularity of $P_{\bf u}$, it is enough to impose  the conditions
$a_{31}=a_{33}=0,$
$-a_{32}b_{31}+a_{41}\neq0,$ and $-a_{32}b_{33}+a_{43}\neq0$, 
in the equation of the $B_2$-curve, yielding the stated expression. \qed
\end{proof}

\medskip

Since the $B_2$-curve along an asymptotic direction ${\bf u}$ coincides with the $A_3$-curve associated with the corresponding binormal direction ${\bf v}$, the $A_4$-point, described in Theorem~\ref{tang A1A2}, is an isolated point on the $B_2$-curve. We therefore have the following result.

\begin{proposition}\label{B2 tran} Let $p$ be a hyperbolic point and a $B_2$-singularity of $P_{ \bf u}$. If $P$ is a versal unfolding of the $B_2$-singularity, then the $B_2$-curve is transverse to the asymptotic directions ${\bf u}$, except when $p$ is an $A_4$-singularity of $h_{\bf v}$ where ${\bf v}$ is a binormal direction associated with ${\bf u}$.
\end{proposition}
\begin{proof} Let ${\bf u}=(0,1,0,0)$ be an asymptotic direction at $p$ associated with the binormal direction ${\bf v}=(0,0,1,0)$ with $p$ a $B_2$-singularity of $P_{\bf u}$. By Proposition \ref{B2curve}, the $B_2$-curve is singular at $p$ if and only if,
$$
-a_{43}+a_{32}a_{31}+a_{32}b_{33}=-4a_{44}+a_{32}^2=0.
$$
This occurs exactly when $h_{\bf v}$ has an $A_4$-singularity at $p$, and the family of orthogonal projections $P$ is not a versal unfolding of the $B_2$-singularity (see Propositions \ref{heigthMonge} and \ref{ClassMonge1}). Since 
$
(4a_{44}-a_{32}^2,-a_{43}+a_{32}a_{31}+a_{32}b_{33})
$
is a tangent vector to the $B_2$-curve, so 
if the family $P$ is a versal unfolding of the $B_2$-singularity, 
the asymptotic direction ${\bf u}$ is transverse to the $B_2$-curve if and only if $-4a_{44}+a_{32}^2\neq0$, and is tangent to the curve otherwise.
 \qed
\end{proof}

\medskip

As a consequence, we obtain a result that relates the singularity of the family of the height function $h$ and the $B_2$-curve.

\begin{corollary}\label{B2regular} Under the above conditions, the family of height functions $h$ is a versal unfolding of the $A_4$-singularity at $p$ if and only if the $B_2$-curve is regular at $p$.
\end{corollary}
\begin{proof}
The proof follows directly from Proposition \ref{B2 tran}.  \qed
\end{proof}

\begin{proposition} \label{S2curve} Let $p$ be a hyperbolic point of $M$ satisfying the conditions of Proposition \ref{ClassMonge1}. Then there exists a neighbourhood of $p$ such that:
\begin{enumerate}
\item[(1)] 
the $S_2$-curve is regular and can be parametrized locally by
$$
3(b_{31}a_{32}-a_{41})x +(3a_{32}a_{30}+2b_{32}a_{32}-2a_{42}+3b_{31}a_{33})y+O(2)=0.
$$
\item[(2)] if $p$ is an $S_3$-singularity of $P_{\bf u}$, the $S_2$-curve can be parametrized locally by
$$
\hspace{-1cm}y= -\frac12\frac{\Lambda_1}{(3a_{32}a_{30}+2b_{32}a_{32}-2a_{42}+3b_{31}a_{33})}x^2+O(3)
$$
where $\Lambda_1=(-10b_{31}a_{32}b_{32}+12a_{32}b_{41}+10a_{42}b_{31}-6b_{31}^2a_{33}+3a_{30}b_{31}a_{32}-12a_{51})$.
\item[(3)] if $p$ is a $C_3$-singularity of $P_{\bf u}$, the $S_2$-curve can be parametrized locally by
$$
3(b_{31}a_{32}-a_{41})x +(3a_{32}a_{30}+2b_{32}a_{32}-2a_{42})y+O(2)=0.
$$
\end{enumerate}
\end{proposition}
\begin{proof} The proof here follows the same steps and arguments as those in Proposition \ref{B2curve}.

(1) We use the following facts. First, $P_{\bf u}$ has an $S_2$-singularity at $p$ if and only if $a_{31}=0$ and $a_{41}-b_{31}a_{32}\neq 0$. 
Finally, considering $P$ as in \eqref{ModP} and writing
$$
(g^1,g^2,g^3)=(x,f^1(x+uy,y)-vy,f^2(x+uy,y)-wy),
$$
the condition for an $S_2$-singularity is given by
\begin{eqnarray*}
  &&-g^2_{x xy} (g^3_{yy})^3 + 2g^3_{x y}(g^3_{y y})^2g^2_{x y y}-2g^3_{x y}g^3_{y y}g^2_{y y}g^3_{x y y}+g^2_{y y}g^3_{x x y}(g^3_{y y})^2 -\\
  &&- (g^3_{x y})^2g^2_{y y y}g^3_{y y}+(g^3_{x y})^2g^2_{y y}g^3_{y y y}=0.
 \end{eqnarray*}
Using these conditions and reverting to the initial coordinate system, we obtain the parametrization of the desired smooth $S_2$-curve.

(2)  If $p$ is an $S_3$-singularity of $P_{\bf u}$, then $a_{31}=b_{31}a_{32}-a_{41}=0$. With these conditions, the $S_2$-curve at the $S_3$-singularity is given by

$
\Big(3a_{32}a_{30}+2b_{32}a_{32}-2a_{42}+3b_{31}a_{33}\Big)y+
$

\hspace{0.8cm}$
\Big(-5b_{31}a_{32}b_{32}+6a_{32}b_{41}+5a_{42}b_{31}-3b_{31}^2a_{33}+\frac32a_{30}b_{31}a_{32}\Big)x^2+\cdots=0$.
Returning to the initial coordinate system, we obtain the desired curve when $P$ is a versal unfolding of the $S_2$-singularity.

(3) If $p$ is a $C_3$-singularity of $P_{\bf u}$, then $a_{31}=a_{33}=0$, $-a_{32}b_{31}+a_{41}\neq 0$ and $-a_{32}b_{33}+a_{43}\neq0$. Substituting it into the parametrization of the $S_2$-curve, we obtain the  
$$
\Big(3a_{32}a_{30}+2b_{32}a_{32}-2a_{42}\Big)y+3\Big(b_{31}a_{32}-a_{41}\Big)x+\cdots$$

\qed
\end{proof}

As a consequence, we obtain the following result.

\begin{theorem}\label{S2regular} Let $p$ be a hyperbolic point and an $S_2$-singularity of $P_{\bf u}$. The $S_2$-curve is regular $p$ if and only if the family of the orthogonal projection $P$ is a versal unfolding of $S_2$-singularity. Moreover, the $S_2$-curve is transverse to one asymptotic direction ${\bf u}$, and the other asymptotic direction is tangent to the $S_2$-curve at the $S_3$-singularity.
\end{theorem}
\begin{proof} Let ${\bf u}=(0,1,0,0)$ and ${\bf u_2}=(1,0,0,0)$ be the asymptotic directions at $p$ and assume that $P_{\bf u}$ has an $S_2$-singularity at $p$. By Proposition \ref{S2curve}, $p$ is a singular point of the $S_2$-curve if and only if
$3(b_{31}a_{32}-a_{41})=0$ and $3a_{32}a_{30}+2b_{32}a_{32}-2a_{42}+3b_{31}a_{33}=0$,
i.e., $p$ is an $S_3$-singularity of $P_{\bf u}$ and the family of $P$ is not a versal unfolding of the $S_2$-singularity. On the other hand, 
$P$ is a versal unfolding of $S_2$-singularity if  $(3a_{32}a_{30}+2b_{32}a_{32}-2a_{42}+3b_{31}a_{33})\neq0$ and
as the tangent vector of the $S_2$-curve at $p$ is given by
$$
\Big((3a_{32}a_{30}+2b_{32}a_{32}-2a_{42}+3b_{31}a_{33}),-3(b_{31}a_{32}-a_{41})\Big),$$ then 
  if $(b_{31}a_{32}-a_{41})\neq0$, then the asymptotic direction ${\bf u}$ is transverse to the $S_2$-curve. If $(b_{31}a_{32}-a_{41})=0$, that is, $p$ is an $S_3$-singularity, the asymptotic direction ${\bf u_2}$ is tangent to the curve. \qed
\end{proof}

\medskip

Although the following result has already been observed in \cite{projsrfR4,IFRT-book}, we present an alternative and simple proof using the parametrizations of the $B_2$-curve and $S_2$-curve.

\begin{proposition}\label{C3trans} Let $p$ be a hyperbolic point and a $C_3$-singularity of $P_{\bf u}$. If the family of the orthogonal projection $P$ is a versal unfolding of the $C_3$-singularity, then the $S_2$-curve and the $B_2$-curve are regular and intersect transversely at the $C_3$-singularity.
\end{proposition}
\begin{proof}
At the $C_3$-singularity at $p$, the 1-jets of the equations of the $B_2$-curve and $S_2$-curve are given, respectively, by
$$
\begin{array}{l}
B_2\text{-curve}: (-a_{43}+a_{32}b_{33})x-(a_{32}^2-4a_{44})y \\
S_2\text{-curve}: 3(b_{31}a_{32}-a_{41})x +(3a_{32}a_{30}+2b_{32}a_{32}-2a_{42})y\\
\end{array}
$$
The curves are transverse if and only if
$$
\begin{array}{|cc|}
 (-a_{43}+a_{32}b_{33})& -(a_{32}^2-4a_{44}) \\
 3(b_{31}a_{32}-a_{41})& (3a_{32}a_{30}+2b_{32}a_{32}-2a_{42})\\
\end{array}\neq0,
$$
which is exactly the condition that guarantees that $P$ is a versal unfolding of the $C_3$-singularity. \qed
\end{proof}

Figure \ref{fig:B2S2} summarizes Theorem \ref{S2regular}, Corollary \ref{B2regular}, and Proposition \ref{C3trans}.

\begin{figure}[htb]
\includegraphics[ width=7.5cm, height=4cm]{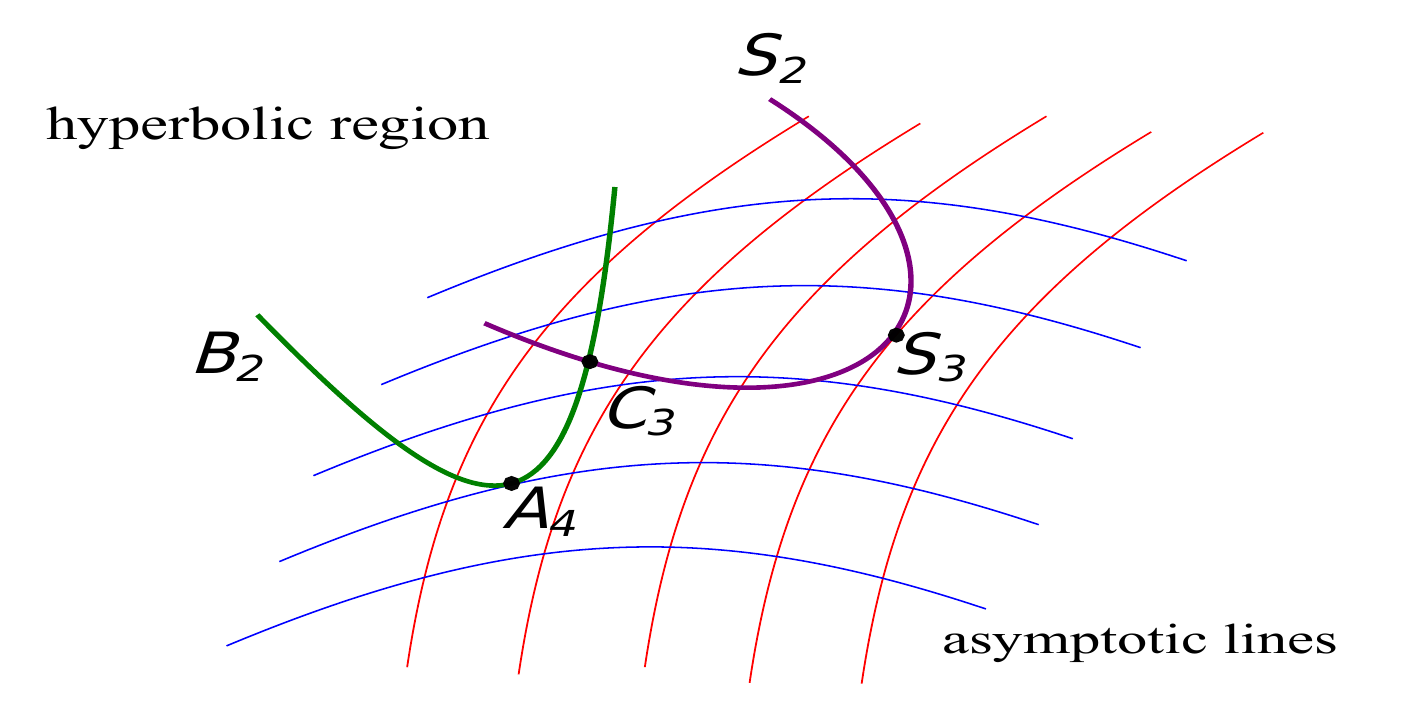}
\caption{The $B_2$-curve and $S_2$-curve at the $A_4$, $C_3$, and $S_3$ singularities on hyperbolic region.}
\label{fig:B2S2}
\end{figure}

\subsubsection{Multi-local robust features}

\

\medskip

In Section \ref{adjacency}, we show that $B_3\rightarrow[A_2]$ and $H_3\rightarrow(A_0S_0)_2$. Here, we describe the curves on $M$ associated with these multi-local singularities.


\begin{theorem}\label{A2-curve} Let $p$ be a hyperbolic point of $M$ satisfying the conditions of Proposition \ref{ClassMonge1}. Suppose that $p$ is a $B_3$-singularity of $P_{\bf u}$. 
Then there exists a neighbourhood of $p$ such that 
the $[A_2]$-curve is regular  and   can be parametrized by
$$
y=-\frac{(a_{31}a_{32}+a_{32}b_{33}-a_{43})}{(a_{32}^2-4a_{44})}x+Bx^2+\cdots,
$$
where the term $B$ depends on the coefficients of the 5-jet of $(f^1,f^2)$. Moreover, $[A_2]$-curve is generically tangent to the $B_2$-curve at the $B_3$-singularity, 
\end{theorem}
 \begin{proof}
 The $P_{\bf u}$ has a $B_3$-singularity if and only if $a_{31}\neq 0,\; a_{33} = 0,$ $4a_{32}^2b^2_{33}+8a_{31}a_{44}b_{33}-2a_{32}a_{43}b_{33}-4a_{31}a_{55}+a^2_{43}= 0$ and $\xi\neq0$ (see Proposition \ref{ClassMonge1}). 
 To simplify the calculations, we perform the change of variables in $P$ as in \eqref{ModP}.
Since $B_3\rightarrow[A_2]$, Proposition \ref{cond [A2]} gives the conditions under which the unfolding $P$ of a $B_3$-singularity exhibits the bigerm $[A_2]$.

Consider $(x_1,y_1)$ and $(x_2,y_2)$ two distinct points. The equality $P(x_1,y_1,u)=P(x_2,y_2,u)$ guarantees that
$$
\begin{array}{l}
  x_1=x_2;\\
 v=a_{32}x_2y_2+a_{31}x_2^2+a_{32}x_2y_{1}+2ux_2+\cdots;\\
w=y_1+y_2+b_{33}y_1^2+b_{33}y_2^2+b_{33}y_1y_2+b_{31}x_2^2+b_{32}x_2y_2+b_{32}x_2y_1+\cdots.\\
\end{array}
$$
Substituting $x,v$ and $w$ into the condition that the two tangent planes at $(x_1,y_1)$ and $(x_1,y_2)$ are equal implies, by the Implicit Function Theorem, that
$$
\begin{array}{l}
 u=\displaystyle-\frac{2 a_{44}}{a_{32}}y_1-\frac{2 a_{44}}{a_{32}}y_2+\frac{(-a_{43}+a_{32} b_{33})}{a_{32}}x_2+\cdots;\\
 y_1=\displaystyle-y_2+\frac{-2( b_{33}a_{32}-a_{43}-a_{31} a_{32})}{-4 a_{44}+ a_{32}^2}x_2+\cdots\\
\end{array}
$$
Finally, substituting all the equalities into
$$
\det(P_x,P_y,P_{xx})\cdot \det(P_x,P_y,P_{yy})-(\det(P_x,P_y,P_{xy}))^2=0
$$ 
we obtain, by the Implicit Function Theorem, that
$$
\begin{array}{l}
 x_2=\displaystyle-\frac{(-4 a_{44}+a_{32}^2)}{(b_{33} a_{32}-a_{43}+a_{31}a_{32})}y_2+\cdots
\end{array}
$$
Returning to the initial coordinate system by the change $(\bar{x},\bar{y})\mapsto (x-uy, y)$, we obtain that the set of points $(p, {\bf u})$ in $M \times S^3$ such that $P_{\bf u}$ has an $[A_2]$-singularity at $p$ is a smooth curve. Projecting it onto $M$, we obtain the desired smooth curve.  

On the other hand, by Proposition \ref{B2curve}, the $B_2$-curve at the $B_3$-singularity can be parametrized by
$$
y=-\frac{(a_{31}a_{32}+a_{32}b_{33}-a_{43})}{(a_{32}^2-4a_{44})}x+Ax^2+\cdots,
$$ 
where $A$ depends on the coefficients of the $5$-jet of $(f^1,f^2)$ with $A\neq B$. 
Therefore, the $[A_2]$-curve and $B_2$-curve  are generically tangent at $B_3$-singularity.
\qed
\end{proof}

\medskip

Since the $H_3$-singularity of $P_{\bf u}$ occurs at a parabolic point $p$, we present now the behaviour of the multi-local curve on the parabolic set around $p$. We first need to give the parametrization of the parabolic curve, which will be useful for the result presented subsequently.

Recall that the parabolic set $\Delta$ is the zero set of the Equation \eqref{ParaCurve}. At parabolic points we can take the Monge form \eqref{Monge} with $j^2(f^1,f^2)=(x^2,xy)$, then
the $2$-jet of the coefficients $a,b,c,l,m$ and $n$ is given by
$$
\begin{array}{l}
a= \frac{1}{2}f^1_{xx}=1+ 3 a_{30} x + a_{31} y +6 a_{40} x^2+ 3a_{41} xy + a_{42} y^2    \\
b= \frac{1}{2}f^1_{x y}=a_{31}x+a_{32}y+\frac32a_{41}x^2+2a_{42}xy+\frac32a_{43}y^2\\
c= \frac{1}{2}f^1_{y y}=a_{32}x+3a_{33}y+a_{42}x^2+3a_{43}xy+6a_{44}y^2\\
l= \frac{1}{2}f^2_{x x}=3b_{30}x+b_{31}y+6b_{40}x^2+3b_{41}xy+b_{42}y^2\\
m= \frac{1}{2}f^2_{x y}=\frac12+b_{31}x+b_{32}y+\frac32b_{41}x^2+2b_{42}xy+\frac32b_{43}y^2\\
n= \frac{1}{2}f^2_{y y}=b_{32}x+3b_{33}y+b_{42}x^2+3b_{43}xy+6b_{44}y^2
\end{array}
$$

Thus, the 2-jet of the parabolic set $\Delta$ is given by
\begin{eqnarray*}
j^2\Delta&=&a_{32}x+3a_{33}y+\Big(b_{32}^2-2a_{31}b_{32}+2a_{32}b_{31}+a_{42}+(2b_{31}+3a_{30})a_{32}\Big)x^2\\
&&+2\Big(3b_{33}b_{32}-3a_{31}b_{33}+3a_{33}b_{31}+\frac32a_{43}+(b_{32}+\frac12a_{31})a_{32}\\
&&+3(b_{31}+\frac32a_{30})a_{33}\Big)xy+3\Big(3b_{33}^2-2a_{32}b_{33}+2a_{33}b_{32}+2a_{44}\\
&&+2(b_{32}+\frac12a_{31})a_{33}\Big)y^2.
\end{eqnarray*}


\begin{proposition}\label{CurvaParabolica} 
With the above notation, if $p$ is an $H_3$-singularity of $P_{\bf u}$, the parabolic set has the following parametrization:
\begin{eqnarray*}
y&=& -\frac{1}{3}\frac{a_{32}}{a_{33}}x-\frac{1}{9a_{33}^3}(6a_{31}a_{32}a_{33}b_{33}-6a_{31}a_{33}^2b_{32}+2a_{32}^2a_{33}b_{32}-3a_{32}a_{33}a_{43}-\\
&&2a_{32}^3b_{33}+3a_{32}^2b_{33}^2-6a_{32}a_{33}b_{32}b_{33}+3a_{33}^2b_{32}^2+2a_{32}^2a_{44}+3a_{33}^2a_{42})x^2+\cdots.
\end{eqnarray*}
\end{proposition}
\begin{proof} 
 The projection $P_{\bf u}$ has an $H_3$-singularity at $p$ if and only if the conditions of Proposition \ref{CondPar} are satisfied. Thus, the expression for the parabolic set $\Delta$ around $p$ follows from the Implicit Function Theorem. \qed
\end{proof}

\begin{theorem}\label{A0S0-curve} 
Let $p$ be a parabolic point of $M$ satisfying the conditions of Proposition \ref{CondPar}. Consider $p$  an $H_3$-singularity of $P_{\bf u}$. 
Then there exists a neighbourhood of $p$ such that 
the $(A_0S_0)_2$-curve is regular and can be parametrized by

\begin{eqnarray*}
y&=&-\frac{a_{32}}{3a_{33}}x- \frac{1}{9a_{33}^{3}}(18a_{31}a_{32}a_{33}b_{33}  + 18a_{31}a_{33}^{2}b_{32}  + 8a_{32}^{3}b_{33}  - 8a_{32}^{2}a_{33}b_{32}  +\\
&&6a_{32}^{2}a_{44}  + 9a_{32}^{2}b_{33}^{2}  - 9a_{32}a_{33}a_{43}  + 18a_{32}a_{33}b_{32}b_{33}  + 9a_{33}^{2}a_{42}  + 9a_{33}^{2}b_{32}^{2}).
 \end{eqnarray*} 
Consequently, the $(A_0S_0)_2$-curve is tangent to the parabolic set at the $H_3$-singularity and has generically contact of order $2$ at $p$.
\end{theorem}

\begin{proof} By Proposition \ref{CondPar}, $P_{\bf u}$ has an $H_3$-singularity at $p$ if and only if $a_{33}\neq0,\;(-b_{32}b_{44}+b_{55})a_{33}^2+((a_{32}b_{44}+a_{44}b_{32}-a_{55})b_{33}-a_{44}b_{44})a_{33}-a_{32}a_{44}b_{33}^2+a_{44}^2b_{33}=0$ and $\zeta\neq0$. Consider the family of projections $P$ with the same change of variables made in \eqref{ModP}.
Since $H_3\rightarrow(A_0S_0)_2$, Proposition \ref{cond(A0S0)2} provides the conditions under which the unfolding $P$ of a $H_3$-singularity exhibits the bigerm $(A_0S_0)_2$:
\begin{itemize}
  \item[(i)] $P(m_1,m_2,u)=P(n_1,n_2,u)$ for $(m_1,m_2)\neq(n_1,n_2)$;
  \item[(ii)] The point $(m_1,m_2)$ is a cross-cap point;
  \item[(iii)] The limiting tangent vector of the double point curve of the cross-cap is contained in the tangent space to the image of $P$ at the point $(n_1,n_2)$.    
\end{itemize}
By condition (ii) we obtain
$$
\begin{array}{l}
   w=m_1+b_{31}m_1^2+2b_{32}m_1m_2+3b_{33}m_2^2+m_2u+\cdots\\
   v= a_{31}m_1^2+2a_{32}m_1m_2+3a_{33}m_2^2+2um_1+\cdots \\
 \end{array}
$$

Substituting $w$ and $v$ into the equality given in (i), the Implicit Function Theorem guarantees that

{\footnotesize
\begin{eqnarray*}
   n_1&=&m_1 \\
   n_2&=& -\frac{a_{32}}{a_{33}}m_1-2m_2+\frac{-(2a_{31}a_{33}-a_{32}^2)}{a_{33}^2}m_1u-\frac{1}{a_{33}}u^2-\frac{2a_{44}a_{32}}{a_{33}^2}m_1m_2\\
   &&-\frac{3a_{44}}{a_{33}}m_2^2
  -\frac{a_{42}a_{33}^2+a_{44}a_{32}^2-a_{43}a_{33}a_{32}}{a_{33}^3}m_1^2+\cdots \\
   u&=&\frac{a_{32}b_{33}-b_{32}a_{33}}{a_{33}}m_1  -\frac{2a_{32}(b_{44}a_{33}-b_{33}a_{44})}{a_{33}^2}m_1m_2-\frac{3(b_{44}a_{33}-b_{33}a_{44})}{a_{33}}m_2^2 \\
 && +\Big(\frac{-b_{42}a_{33}^3-2a_{32}b_{33}^2b_{32}a_{33}+2a_{32}^2b_{33}b_{32}a_{33}-a_{32}^2b_{44}a_{33}+b_{33}a_{42}a_{33}^2}{{a_{33}^3}}\\
  && +\frac{-b_{33}a_{43}a_{32}a_{33}-2b_{32}b_{33}a_{31}a_{33}^2+2b_{32}b_{31}a_{33}^3-2a_{32}b_{33}b_{31}a_{33}^2+a_{32}^2b_{33}^3}{a_{33}^3}\\
  &&  +\frac{-b_{32}^2a_{32}a_{33}^2+a_{32}b_{43}a_{33}^2-a_{32}^3b_{33}^2+2a_{32}b_{33}^2a_{31}a_{33}+b_{33}a_{44}a_{32}^2+b_{32}^2b_{33}a_{33}^2}{a_{33}^3}\Big)m_1^2\\
 \end{eqnarray*}
}

To find the limiting tangent vector of the double point curve of the cross-cap, in order to simplify the calculations, we need to move the singularity $(m_1,m_2)$ to the origin by a translation and look for the image of the double points. Thus, let us consider $P(x+m_1,y+m_2)$. Searching for the double point curve, we have
$P(x_1+m_1,y_1+m_2)=P(x_2+m_1,y_2+m_2)$ with $(x_1+m_1,y_1+m_2)\neq(x_2+m_1,y_2+m_2)$ if and only if, we obtain
$$
\begin{array}{l}
x_1=c_1y_2+c_2y_2^2+\cdots\\ 
y_1=-y_2+d_2y_2^2+\cdots,\\
\end{array}
$$
where $c_1,c_2,d_2$ are functions of the variables $m_1,m_2,u$. Substituting $x_1$ and $y_1$ into $P(x_1+m_1,y_1+m_2)$ we obtain the double point curve parametrized by
$$
\alpha(y)=(\xi_1(y,m_1,m_2),\xi_2(y,m_1,m_2),\xi_3(y,m_1,m_2)),
$$
where $\xi_i$ is a function germ in $y,m_1$ and $m_2$. Therefore, the limit of the tangent vector of the curve at the point $0$ gives
$$
\alpha''(0)=(\xi''_1(m_1,m_2),\xi''_2(m_1,m_2),\xi''_3(m_1,m_2)),
$$
where
\begin{eqnarray*}
  \xi''_1(m_1,m_2)&=&2c_2-u(2d_1-\frac{2}{3}\frac{a_{32}c_2}{a_{33}})+\cdots \\
   \xi''_2(m_1,m_2)&=& 2c_1^2+(6a_{30}c_1^2-4a_{31}c_1+2a_{32}+4c_2)m_1\\
   &&+(2a_{31}c_1^2-4a_{32}c_1+6a_{33})m_2+\cdots \\
   \xi''_3(m_1,m_2)&=&-2c_1+(-4b_{31}c_1+6b_{30}c_1^2+\frac{2}{3}\frac{a_{32}c_2}{a_{33}+2b_{32}})m_1\\
   &&+(2b_{31}c_1^2-4b_{32}c_1+6b_{33}+2c_2)m_2+\cdots
\end{eqnarray*}


Since $\alpha''(0)$ belongs to the tangent plane to the image of $P$ at the point $(n_1,n_2)$, it follows that
$$
\frac{2c_1(a_{32}m_1+3a_{33}m_2+\cdots)}{a_{33}}=0.
$$
Therefore, by the Implicit Function Theorem, $m_2=-\frac{1}{3}\frac{a_{32}}{a_{33}}m_1+\cdots$. Substituting into the expressions of $u,v,w$ and $n_2$, we obtain a smooth curve in $M\times S^3$ parametrized by

\begin{eqnarray*}
  n_1&=&m_1 \\
  n_2&=& -\frac{1}{3}\frac{a_{32}}{a_{33}}m_1+\frac{1}{3a_{33}^3}(3a_{32}^2a_{33}b_{32}-3b_{33}a_{32}^3-6b_{33}a_{33}b_{32}a_{32}-6a_{33}^2a_{31}b_{32}\\
&&+2a_{44}a_{32}^2+3a_{33}^2a_{42}+6a_{32}a_{31}b_{33}a_{33}+3a_{33}^2b_{32}^2-3a_{33}a_{43}a_{32}+3b_{33}^2a_{32}^2)m_1^2\\
  u&=&\left(\frac{-b_{32}a_{33}+a_{32}b_{33}}{a_{33}}\right)m_1+\cdots\\
   v&=& \displaystyle\left(\frac{-a^2_{32}+6a_{32}b_{33}-6a_{33}b_{32}+3a_{33}a_{31}}{3a_{33}}\right)m^2_1+\cdots\\
   w&=&\displaystyle m_1-\left(\frac{1}{3}\frac{(b_{33}a_{32}^2-3 b_{31}a_{33}^2+3b_{31}a_{33}^2)}{a^2_{33}}\right)m^2_1+\cdots
\end{eqnarray*}

 Returning to the initial coordinate system and projecting by $\pi_1:M\times S^3\rightarrow M$, we obtain the desired smooth $(A_0S_0)_2$-curve. 
%
%

By Proposition \ref{CurvaParabolica}, at the $H_3$-singularity the curve $\Delta$ is parametrized by
\begin{eqnarray*}
y&=& -\frac{1}{3}\frac{a_{32}}{a_{33}}x-\frac{1}{9a_{33}^3}(6a_{31}a_{32}a_{33}b_{33}-6a_{31}a_{33}^2b_{32}+2a_{32}^2a_{33}b_{32}-3a_{32}a_{33}a_{43}-\\
&&2a_{32}^3b_{33}+3a_{32}^2b_{33}^2-6a_{32}a_{33}b_{32}b_{33}+3a_{33}^2b_{32}^2+2a_{32}^2a_{44}+3a_{33}^2a_{42})x^2+\cdots.
\end{eqnarray*}
Therefore, the $(A_0S_0)_2$-curve and parabolic curve $\Delta$ are tangent at the $H_3$-singularity with contact of order $2$.
\end{proof}

We summarize Theorem \ref{tang A1A2}, \ref{A2-curve} and \ref{A0S0-curve} in Figure \ref{fig:NewSpecialC}.

\begin{figure}[htb]
\includegraphics[ width=11.5cm, height=7cm]{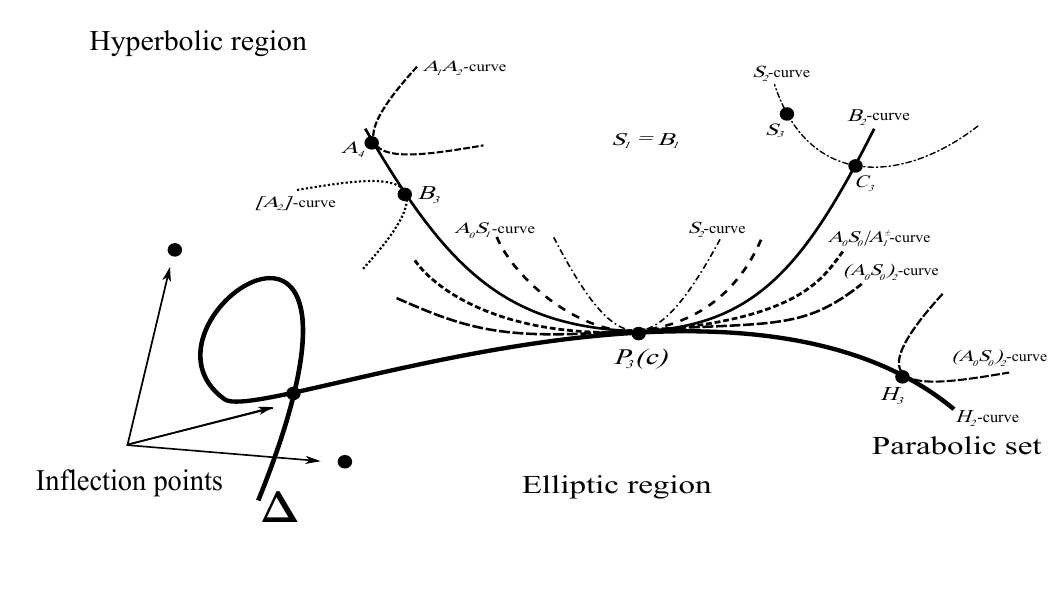}
\caption{Curves and special points on $M$ with multi-local curves. In Figure \ref{fig:SpecialC}, we have added the $A_1A_2$-curve at $A_4$-point, $[A_2]$-curve at $B_3$-point, and $(A_0S_0)_2$-curve at $H_3$-point.}
\label{fig:NewSpecialC}
\end{figure}



\section{Declarations}
\hspace{-0,4cm}{\bf Conflicts of interest:} The author declare no conflict of interest.


\end{document}